\documentclass[a4paper]{amsart}

\usepackage{cite}
\usepackage{hyperref}
\usepackage{enumerate}
\usepackage{graphicx}
\usepackage{amsthm}
\usepackage{amsmath,amsfonts,amssymb}
\usepackage{graphics,color}
\usepackage{color}		
\usepackage{epsfig}
\usepackage{mathrsfs}
\usepackage{sidecap}

\usepackage{caption}
\usepackage{subcaption}

\usepackage{color}
\usepackage{cases}
\usepackage{esint}

\theoremstyle{plain}
\newtheorem{theorem}{Theorem}[section]
\newtheorem{lemma}[theorem]{Lemma}
\newtheorem{proposition}[theorem]{\bf Proposition}
\newtheorem{corollary}[theorem]{Corollary}
\theoremstyle{remark}
\newtheorem{definition}[theorem]{\bf Definition}
\newtheorem{remark}[theorem]{\bf Remark}

\numberwithin{equation}{section}

\DeclareMathOperator*{\conv}{conv}

\newcommand*{\R}{\ensuremath{\mathbb{R}}}
\renewcommand*{\S}{\ensuremath{\mathcal{S}}}
\newcommand*{\N}{\ensuremath{\mathbb{N}}}
\newcommand*{\Z}{\ensuremath{\mathbb{Z}}}

\def\rn#1{\mathbb{R}^{#1}}

\newcommand*{\Q}{\ensuremath{\mathcal{Q}}}
\renewcommand*{\O}{\ensuremath{\mathcal{O}(1)}}
\renewcommand*{\L}{\ensuremath{\mathcal{L}}}

\newcommand*{\ul}{\ensuremath{u^\mathrm{L}}}
\newcommand*{\ur}{\ensuremath{u^\mathrm{R}}}
\newcommand*{\unu}{\ensuremath{u^\nu}}

\newcommand*{\e}{\ensuremath{\epsilon}}

\newcommand*{\BV}{\ensuremath{\text{BV}}}

\newcommand{\tv}[1]{\ensuremath{\mathrm{Tot.Var.}{\left\{#1\right\}}}}

\newcommand*{\Om}{\ensuremath{\Omega}}
\newcommand{\mres}{\mathbin{\vrule height 1.6ex depth 0pt width 0.13ex\vrule height 0.13ex depth 0pt width 1.3ex}}
\renewcommand*{\H}{\ensuremath{\mathcal{H}}}
\renewcommand*{\S}{\ensuremath{\mathrm{S}}}
\newcommand*{\iunu}{\ensuremath{u_0^\nu}}
\newcommand*{\iu}{\ensuremath{u_0}}
\newcommand{\infnorm}[1]{\ensuremath{{\|}}{{#1}}\ensuremath{\|_\infty}}
\newcommand*{\fnu}{\ensuremath{f^\nu}}

\newcommand*{\tif}{\ensuremath{\text{if }}}
\newcommand*{\rp}{\ensuremath{\mathbb{R}^+}}

\renewcommand*{\ss}{\ensuremath{\mathcal{S}(s)}}

\newcommand*{\G}{\text{Graph}}

\newcommand*{\intinf}{\ensuremath{\int^{\infty}_{-\infty}}}

\newcommand*{\sgn}{\ensuremath{\text{sgn}}}
\newcommand*{\tas}{\ensuremath{\ \text{as }}}

\newcommand*{\tfor}{\ensuremath{\ \text{for }}}

\newcommand*{\tforall}{\ensuremath{\ \text{for all }}}
\newcommand*{\br}{\ensuremath{B_{\rho}}}
\newcommand*{\pse}{\ensuremath{\partial^* E}}

\newcommand*{\pt}{\ensuremath{\frac{\partial}{\partial t}}}
\newcommand*{\dt}{\ensuremath{\frac{d}{dt}}}

\newcommand{\Osc}[1]{\ensuremath{\text{Osc.}{\left\{#1\right\}}}}
\newcommand*{\sty}{\ensuremath{\mathcal{S}_{\tau,\eta}}}

\renewcommand*{\conv}{\ensuremath{\text{conv}}}

\newcommand*{\loc}{\ensuremath{\text{loc}}}
\newcommand*{\cancel}{\ensuremath{\mathrm{Cancel}}}
 \newcommand*{\id}{\ensuremath{\mathrm{Id}}}

 \newcommand*{\opep}{\ensuremath{\overline{P_\e P}}}
  
   \newcommand*{\opebe}{\ensuremath{\overline{P_\e B_e}}}
\newcommand{\osc}[1]{\ensuremath{\mathrm{osc}\{#1\}}}

\begin{document}

\title[Structure of entropy solutions to scalar conservation laws]{Structure of entropy solutions to general scalar conservation laws in one space dimension}


\author{Stefano Bianchini}
\address{SISSA, via Bonomea 265, 34136 Trieste, ITALY}
\email{bianchin@sissa.it}
\urladdr{http://people.sissa.it/~bianchin}

\author{Lei Yu}
\address{SISSA, via Bonomea 265, 34136 Trieste, ITALY}
\email{yulei@sissa.it}


\begin{abstract}
In this paper, we show that the entropy solution of a scalar conservation law is
\begin{itemize}
\item continuous outside a $1$-rectifiable set $\Xi$,
\item up to a $\mathcal H^1$ negligible set, for each point $(\bar t,\bar x) \in \Xi$ there exists two regions where $u$ is left/right continuous in $(\bar t,\bar x)$.
\end{itemize}
We provide examples showing that these estimates are nearly optimal.

In order to achieve these regularity results, we extend the wave representation of the wavefront approximate solutions to entropy solution. This representation can the interpreted as some sort of Lagrangian representation of the solution to the nonlinear scalar PDE, and implies a fine structure on the level sets of the entropy solution.
%
%
\end{abstract}

\thanks{This work is supported by the ERC Starting Grant 240385 "ConsLaw"}

\maketitle

\centerline{Preprint SISSA 11/2014/MATE}

\tableofcontents

\section{Introduction}
\label{S_intro}

This paper is concerned with the pointwise structure of entropy solutions to the Cauchy problem of scalar conservation law
\begin{subequations}
\label{e:cauchy}
\begin{equation}
\label{e:cauchyscalar}
u_t+f(u)_x=0,
\end{equation}
\begin{equation}
\label{e:cauchyinitial}
u(0,x)=u_0(x),
\end{equation}
\end{subequations}
assuming that $f\in C^2(\R)$ and $\iu\in L^1$ has bounded total variation. It is well known that entropy solutions satisfies
\begin{equation*}
\intinf |u(t,x)-v(t,x)|dx \leq \intinf |u(0,x)-v(0,x)|dx,
\end{equation*}
for each $t\geq 0$. This inequality implies the uniqueness and continuous dependence of the entropy solution (for reference on hyperbolic systems in one space dimensions, see \cite{Bre}). The above estimate yields also that
\begin{equation*}
\tv{u(t),\R} \leq \tv{u(0),\R},
\end{equation*}
and by means of the PDE one obtains in the sense of measures
\begin{equation*}
u_t = - f(u)_x.
\end{equation*}
This implies that $u$ is a $\BV_{\loc}$ function of the two variables $t,x$. Hence, it shares the pointwise structure properties of BV functions, that is, $u$ either is approximately continuous or has an approximate jump at each point $(t,x)$ out of a set $\Theta$ whose Hausdorff 1-dimensional measure is zero. (See Section 5.9 of \cite{EG}).

As the entropy solution $u$ can be constructed as the limit of front tracking approximate solutions $\{\unu\}_{\nu\geq 1}^\infty$, it is possible to study the pointwise structure by analyzing the corresponding properties of the approximations.  Then, passing to the limit, one can get the desired properties of the entropy solution.

This strategy has been used for admissible BV solutions to hyperbolic systems of conservation laws in \cite{BL} by Bressan and LeFloch. They proved that if the solution is obtained as the limit of front tracking approximate solutions, then the jump discontinuity points of the solution focus on countably many Lipschitz curves and outside these curves and countably many irregular points, the solution is continuous. Moreover, outside the irregular points, the solution has left and right limits along these curves and the speed of these curves satisfy the Rankine-Hugoniot relations. Later, in \cite{Bre}, Bressan proves the same result for the case when the characteristic fields of the equation are linearly degenerate or genuinely nonlinear. In \cite{BY-global}, Bianchini and Yu generalized this structure result to piecewise genuinely nonlinear systems.

However, in general, the same global structure result is not true for strictly hyperbolic system. In fact, in \cite{Bre-global} (see also Section \ref{S_examples}), the author constructs a $2\times 2$ strictly hyperbolic system
whose characteristic fields are neither piecewise genuinely nonlinear nor linearly degenerate. One can choose the initial data $u_0$ properly in order to have an admissible BV solution such that the set of its jump discontinuities belonging to the second characteristic family does not contain any segment on the domain $[0,1]\times\R$, even if it is of positive $\mathcal H^1$-measure. 

This paper contains two main results.

The first one, contained in Section \ref{S_pointwise_regul}, is that we prove a generalization of the regularity results to entropy solutions of scalar equations with general flux: more precisely,
\begin{enumerate}
\item the exceptional set $\Theta$ is at most countable (see \eqref{E_count_Theta});
\item $u$ is continuous outside countably many Lipschitz curves $\Xi$ (Lemma \ref{L_jump_set} and Proposition \ref{P_cont_out_jumps});
\item the exceptional set is contained in $\Xi$ (Lemma \ref{L_Theta_on_Xi});
\item in $\Xi \setminus \Theta$ the solution $u$ has $L^1$-approximate right and left limits (Theorem \ref{T_rect_out_inter}). 
\end{enumerate}
These results are the generalization of analog results for genuinely nonlinear systems, given in \cite{BL}, and extended in \cite{Liu_adm} to systems with finitely many inflection points on each characteristic field. \\
In \cite{BL} it is also shown that outside $\Theta$ there are right and left limits on the curves in $\Xi$. We are not able to recover completely this regularity, and in fact in the last section we present counterexamples showing that this is not true. Our best result (Theorem \ref{T_regul_jumps}) is that for every point $(\bar t,\bar x) \in \Xi$ there are two curves $\gamma^-_{\bar t,\bar x}(t) \leq \gamma^+_{\bar t,\bar x}(t)$ such that
\begin{equation*}
\lim_{\substack{(t,x) \to (\bar t,\bar x) \\ x < \gamma^-_{\bar t,\bar x}(t)}} \big| u(t,x) - u(\bar t,\bar x-) \big| = 0, \quad \lim_{\substack{(t,x) \to (\bar t,\bar x) \\ x > \gamma^+_{\bar t,\bar x}(t)}} \big| u(t,x) - u(\bar t,\bar x+) \big| = 0.
\end{equation*}
Moreover Theorem \ref{T_rect_out_inter} shows that outside a countable set $\Theta$ and for all $\delta > 0$ it holds
\begin{equation*}
\lim_{\substack{(t,x) \to (\bar t,\bar x) \\ x < \bar x + p(t-\bar t) + \delta |t - \bar t|}} u(t,x) = u(\bar t,\bar x-), \qquad \lim_{\substack{(t,x) \to (\bar t,\bar x) \\ x > \bar x + p(t-\bar t) + \delta |t - \bar t|}} u(t,x) = u(\bar t,\bar x+)
\end{equation*}
In particular these points are of approximate jump.

In order to prove these results, we study the level set structure of (the right-continuous representative of) the entropy solution $u$. To this aim, we generalize the construction of \cite{BM} for wavefront/Glimm approximated solutions to the entropy solution. This representation of the solution done in Section \ref{Ss_lagr_repr_real_u} is the second main result of the paper.

Set $J:=[0,\tv{u_0}]$. By \emph{wave representation} of the entropy solution $u^\nu$ we mean the unique triple of functions $(\mathtt X,\mathtt u,\mathtt a)$,
%
\begin{equation*}
\begin{array}{rclll}
\mathtt X &:& \R^+ \times J \supset E \to \R, & & \text{the \emph{position of the wave $s$}}, \\
\mathtt u &:& J \to \R, & & \text{the \emph{value of the wave $s$}}, \\
\mathtt a &:& \rp \times J \to \{-1,0,1\}, & & \text{the \emph{signed existence interval of the wave $s$}},
\end{array}
\end{equation*}
satisfying the following conditions:
\begin{enumerate}
\item \label{Cond_1_intro} the function $\mathtt a$ is of the form
\begin{equation*}
\mathtt a(t,s) = \mathcal S(s) \chi_{[0,\mathtt T(s))}(t)
\end{equation*}
for some functions
\begin{equation*}
\begin{array}{rclll}
\mathcal S &:& J \to \{-1,1\}, & & \text{the \emph{sign of the wave $s$}}, \\
\mathtt T &:& J \to \R^+, & & \text{the \emph{time of existence of the wave $s$}};
\end{array}
\end{equation*}
\item the set $E$ is given by
\begin{equation*}
E := \big\{ (t,s) : \mathtt T(s) \geq t \big\};
\end{equation*}
\item \label{Cond_2_intro} $s \mapsto \mathtt X(t,s)$ is increasing for all $t$, $t \mapsto \mathtt X(t,s)$ is Lipschitz for all $s$, and 
\begin{enumerate}
\item $D_x u(t) = \mathtt X(t,\cdot)_\sharp \Big( a(t,\cdot)\L^1\mres J \Big)$; 
\item $|D_x u|(t) = \mathtt X(t,\cdot)_\sharp \Big( |a(t,\cdot)| \L^1\mres J \Big)$;
\end{enumerate}
\item \label{Cond_3_intro} the value $\mathtt u$ 
is a $1$-Lipschitz function of $s$ satisfying 
\begin{equation*}
\frac{d}{ds} \mathtt u(s) = \mathcal S(s) \qquad \text{for \ $\mathcal L^1$-a.e. $s$}.
\end{equation*}
\end{enumerate}
The study of these functions for the approximate solution $u^\nu$ is fundamental to prove the main quadratic interaction estimate in \cite{BM}. Here for completeness we present a sketch of the proof of the existence of the wave representation for piecewise constant wavefront approximate solutions $u^\nu$ (Proposition \ref{P_uniq_wave_repr}) and of the quadratic interaction estimate (Theorem \ref{T_quadratic}).

To repeat this construction for entropy solutions, as a first step we prove a uniform bound on the level sets of the approximate solutions $u^\nu$: up to a negligible set of values, the set $\{u^\nu > w\}$ is bounded by finitely many Lipschitz curves $\{\gamma^\nu_{j,w}\}_j$, whose number depends only on the initial data \eqref{e:cauchyinitial}, and whose Lipschitz constant depends only on the flux $f$ of \eqref{e:cauchyscalar} (Lemma \ref{L_comap_A_nu_iw}). This fact follows from the particular choice of initial data for $u^\nu$ (Lemma \ref{l:approxbv}) and elementary properties of the wavefront tracking algorithm.

It is fairly easy to see that $\{u^\nu > w\}$ are thus (locally) precompact in Hausdorff topology, and due to the $L^1$-convergence of $u^\nu$ to the entropy solution $u$, one concludes that the same level set structure is true for $u$: there exists a representative $u$ of the entropy solution such that the level sets $\{u > w\}$ are bounded by finitely many Lipschitz curves $\{\gamma_{j,w}\}_j$ (Theorem \ref{T_right_cts_u}). An argument based on the monotonicity of the semigroup generated by \eqref{e:cauchyscalar} yields that one can restrict the intervals of existence of the curves in order to have them disjoint (Corollary \ref{C_wave_repr}).

At this point, the definition of the functions $(\mathtt X,\mathtt u,\mathtt a)$ is only a matter of reparametrizing the functions
\begin{equation*}
(j,w) \mapsto (\gamma_{j,w},w,(-1)^{j+1}),
\end{equation*}
where the latter is the sign of the $x$-derivative of $u$ on $\gamma_{j,w}$, assuming $j \mapsto \gamma_{j,w}$ ordered: this is done in Proposition \ref{P_map_s_jw} and Theorem \ref{T_wave_repr_cont}.

\subsection{Structure of the paper}
\label{Ss_paper_struct}

The paper is organized as follows.

In Section \ref{s:bv}, we recall some classical results about BV functions, e.g.  \emph{coarea formula} (Theorem \ref{t:cf}) and a special approximation lemma (Lemma \ref{l:approxbv}).

Section \ref{s:fta} also contains results already present in the literature. We include them for reference and also because some proofs are simpler than in the published works, and in order to underline the fact that these bounds depends only on $u_0$ and $f$. \\ First of all, after recalling the wavefront algorithm (Section \ref{s:cfta}), we generalize the notion of wave representation which is suitable also for entropy solutions (Section \ref{s:WR}), and we sketch the proof of its existence for wavefront approximate solutions. \\
In Section \ref{s:BM} we present a shortened proof of the quadratic interaction estimate found in \cite{BM}: this estimate implies that the speed of each wave $s$ is BV in time, with a bound depending only on $\tv{u_0}$. \\
The last part (Section \ref{s:elsa}) study the structure of the level sets of the approximate solutions $u^\nu$. Even if this structure is in some sense trivial, the aim is to show some uniform estimates of the boundary of these sets.

These uniform bounds are finally exploited in Section \eqref{S_level_1}. \\
The first result is the simultaneous convergence in Hausdorff and $L^1$ metric of $\mathcal L^1$-a.e. level sets, which shows that each level set $\{u > w\}$ is bounded by countably many disjoint Lipschitz curves (Section \ref{s:clsfta}). \\
Next, in Section \ref{Ss_lagr_repr_real_u}, a simple change of variable allows to build the wave representation of the solution $u$ (Theorem \ref{T_wave_repr_cont}), and to prove some useful estimates on the set where the waves $s$ are canceled (only on the jump set of $u$, Corollary \ref{C_canc_on_jumps}) and on the speed of the waves $s$ ($t \mapsto \mathtt X(t,s)$ is a characteristic of the PDE, Proposition \ref{P_cancell_points}). Finally, the quadratic interaction estimate is shown to be valid for the entropy solution in Theorem \ref{T_quadr_contin}.

The last section, Section \ref{S_pointwise_regul}, uses the wave representation to show the fine regularity structure of the solution $u$. First, it is possible to give a precise estimate on the tame variation/oscillation estimates, classically used in the dependency region of an interval: this is done in Lemmas \ref{L_exact_canc}, \ref{L_exact_l_infty_decr}, and in Corollary \ref{C_triangl_dep_tv} it is shown how to recover the standard estimates. Next, one defines the discontinuity set $\xi$ and the set where $x \mapsto u(t,x)$ is discontinuous, and shows immediately that
\begin{enumerate}
\item $\Xi$ is $1$-rectifiable (Lemma \ref{L_jump_set}),
\item $u$ is continuous outside $\Xi$ (Proposition \ref{P_cont_out_jumps}).
\end{enumerate}
The analysis of the waves $s$ passing through a point $(\bar t,\bar x)$ implies that $u$ is left/right continuous in $(\bar t,\bar x)$ in regions bounded by some Lipschitz curves $\gamma^\pm_{(\bar t,\bar x)}$ (Theorem \ref{T_regul_jumps}): this holds also in the set $\Theta$ where strong interactions/cancellations occurs. Two refinements of this result are then presented:
\begin{enumerate}
\item if $(\bar t,\bar x) \in \Xi \setminus \Theta$, then $(\bar t,\bar x)$ is an $L^1$-approximate jump point of $u$ (Theorem \ref{T_rect_out_inter});
\item for all $\delta > 0$ and $(\bar t,\bar x) \in \Xi \setminus \Theta$, $u(t,x)$ is left/right continuous in the cone regions
\begin{equation*}
\big\{ x < \bar x + \tilde \lambda(\bar t,\bar x) (t - \bar t) - \delta |t-\bar t| \big\} \quad \text{and} \quad \big\{ x > \bar x + \tilde \lambda(\bar t,\bar x) (t - \bar t) + \delta |t-\bar t| \big\},
\end{equation*}
where $\tilde \lambda(\bar t,\bar x)$ is the speed of the jump computed by the Rankine-Hugoniot condition.
\end{enumerate}
We conclude the paper with Section \ref{S_examples}, where we show that these results are in some sense optimal. In fact, we first recall the example of \cite{BY-global} showing that the jump set is a Cantor set $\mathcal J$ of positive $\mathcal H^1$-measure for which $\gamma^\pm_{(\bar t,\bar x)}(t)$, $(\bar t,\bar x) \in \mathcal J$, do not coincide in every open interval. \\
The second example (Section \ref{Sss_canc_vs_inter}) shows that the there may be a strong cancellation with negligible change in speed of the surviving waves. \\
Finally, the example in Section \ref{Sss_not_jump} shows that even if $(\bar t,\bar x) \notin \Theta$ the curves $\gamma^\pm_{(\bar t,\bar x)}(t)$ may not be tangent.

\section{Preliminary results on BV functions}
\label{s:bv}

In this section, we recall some necessary background materials on BV functions: we give only the statements without proof, since they can be easily found in the literature. We denote by $Dv$ the distributional derivative of the function $v$. 


Let $E$ be an $\L^n$-measurable subset of $\rn n$ which is of finite perimeter in the open set $\Om\subset \rn n$; we denote by $\partial^*E$ the reduced boundary of $E$ and by $\chi_E$ the characteristic function of the set $E$. 
\begin{definition}
Let $y\in \rn n$, we say $y\in \partial^* E$, the \emph{reduced boundary} of $E$, if
\begin{enumerate}[(i)]
\item $|D\chi_{E}|(B_\rho(y))>0,\tforall \rho > 0$,
\item $\pi_E(y):= \underset{\rho \searrow 0}{\lim}\frac{D\chi_E(\br (y))}{|D\chi_E|(\br (y))}$ \ \text{exists},
\item $|\pi_E(y)|=1$.
\end{enumerate}
The function $\pi_E:\partial^* E\to\S^{n-1}$ is called the \emph{generalized inner normal} to E.
\end{definition}

\begin{remark}\label{r:trc}
The following facts are well known.
\begin{enumerate}
\item If $E$ is a subset of $\R$ and has finite perimeter in the interval $(a,b)$, then its reduced boundary in $(a,b)$ consists of finitely many points.     
\item It holds (see Section 5.7 in \cite{EG})
\[
|D\chi_E|(\rn n-\partial^* E)=0,
\]
\item $\pse$ is a countably $(n-1)$-rectifiable set and the measure $|D\chi_E|$ coincides with $\H^{n-1} \mres \partial^* E$, see Section 3.5 in \cite{AFP}.
\item If $E \subset \R^2$ and $\partial E$, the topological boundary of $E$, consists of countably many injective Lipschitz curves $\{\gamma_j\}_j$, with $\mathcal H^1( \mathrm{Graph}(\gamma_j) \cap \mathrm{Graph}(\gamma_{j'})) = 0$ for $j \not= j'$, then it coincides with the reduced boundary of $E$ up to a $\mathcal H^1$-negligible set, that is
\[
\mathcal H^1 \big( \partial E \setminus \partial^* E \big) = 0,
\]
and the generalized inner normal to the boundary is equal to the inner normal to the boundary if it exists, i.e. $\mathcal H^1 \mres \partial E$-a.e..
\end{enumerate}
\end{remark}

A BV function in the open set $\Omega$ is an integrable function $v : \Omega \to \R$ such that its distributional derivative is a bounded measure. We recall the following coarea formula (see Theorem 3.40 and Definition 3.60 in \cite{AFP}): for notational convenience, we will write
\begin{equation*}
E_w(v) := v^{-1}\big( (w,+\infty) \big).
\end{equation*}

\begin{theorem}[Coarea formula in BV]
  \label{t:cf}
If $v\in \BV(\Om)$, the set $E_w(v)$ has finite perimeter in $\Om$ for $\L^1$-a.e. $w\in\R$ and
\begin{equation*}
    |Dv|(B)=\intinf\H^{n-1}\mres \partial^* E_w(v)(B)dw,
    \end{equation*}
    \begin{equation}
    \label{e:cfss}
    Dv(B)=\intinf \bigg[ \int_{B\cap \partial^*E_w(v)}\pi_{E_w(v)}(y)\H^{n-1}(dy) \bigg] dw,
\end{equation}
for each Borel set $B\subset\Om$.
\end{theorem}

Consider now an interval $I\subset\R$ and a map $v:\R \to \R^N$ with $N\geq 1$. We denote by $\tv{v; I}$ the \emph{Total Variation of $v$ on the interval $I$}, defined as
\begin{equation*}
\tv{v;I} = \sup_{\underset{x_i < x_{i+1}}{x_i \in I}} \sum_i \big| v(x_{i+1}) - v(x_i) \big|.
\end{equation*}
It is a well known fact that $\tv{v;I} = |Dv|(I)$ and there exists an $L^1$-representative $\tilde v$ such that $\tv{\tilde v;I} = |Dv|(I)$: in this paper we will consider right-continuous functions, so that our functions will be defined everywhere and $\tv{v;I} = |Dv|(I)$.

To obtain  the regularity results though the wavefront tracking approximation, we will use the following family of piecewise constant initial data, see Lemma 2.2 in \cite{Bre} for the construction.

\begin{lemma}
\label{l:approxbv}
Let $v:\R\to\R^N$ be right continuous with bounded total variation. Then, for each $\nu \in\N$, there exists a piecewise constant function $v^\nu$ such that
\begin{equation*}
\|v^\nu\|_\infty \leq \|v\|_\infty, \qquad \| v^\nu - v \|_\infty < 2^{-\nu},
\end{equation*}
\begin{equation*}
\mathcal H^0(\partial \{v^\nu > w\}) \leq \big|D \chi_{\{v > w\}} \big|.
\end{equation*}
\end{lemma}

In particular by coarea formula
\begin{equation}
\label{e:approxbv}
\tv{v^\nu}\leq \tv{v}. 
\end{equation}

\section{Front tracking approximations} 
\label{s:fta}

In order to prove the existence of a \emph{wave representation} (or \emph{Lagrangian representation}) of the solution $u(t)$, we will use the well known fact that the entropy solution $u(t)$ to \eqref{e:cauchy} can be constructed as the limit of \emph{wavefront tracking approximations $\{\unu(t)\}_{\nu\geq 1}$}. We first briefly recall the wavefront tracking algorithm, with the aim of settling the notation, and then we present a simplified proof of the existence of a \emph{wave representation} for the approximate solution and an additional compactness estimate for the speed of the wavefront: the original proofs are given in \cite{BM}. We finally conclude this section with some analysis of the structure of level sets for wavefront approximate solutions, which will be useful later.

\subsection{The construction of front tracking approximations}
\label{s:cfta}


By Lemma \ref{l:approxbv} and \eqref{e:approxbv}, we consider a sequence $\{\iunu\}_{\nu\geq 1}$ of right continuous piecewise constant functions with finite jump discontinuities, such that
\begin{itemize}
\item[(1)] $\iunu(x)\in 2^{-\nu}\Z,\ \forall\, x\in \R$,
\item[(2)] $\|\iunu-\iu\|_{L^\infty} \to 0$,
\item[(3)] $\tv{\iunu}\leq\tv{\iu}$,
\item[(4)] $\infnorm{\iunu}\leq \infnorm{\iu}$.
\end{itemize}
Let $\fnu$ be the piecewise affine interpolation of $f$ with grid points of $2^{-\nu} \N$:
\[
\fnu(u) = (1-t) f\big(2^{-\nu} k\big) + t f\big(2^{-\nu}(k+1)\big) \quad \text{for} \quad u = (1-t) 2^{-\nu} k + t 2^{-\nu}(k+1), \ t \in [0,1].
\]

For a fixed $\nu\geq 1$, let $x_1<\cdots<x_p$ be the points where $\iunu$ is discontinuous: at each $x_i$, solve the Riemann problem
\begin{equation*}
 \begin{cases}
  (\tilde u_i)_t+(f^\nu(\tilde u_i))_x=0,\\
\tilde u_i(0,x)=\begin{cases}
                \iunu(x_i-) &\tif x<0,\\
                \iunu(x_i+) &\tif x>0.
                \end{cases}
\end{cases}
\end{equation*}
Recall that the solution $\tilde u_i$ is given by
\[
\tilde u_i(t,x) =
\begin{cases}
\Big( \frac{d}{du} \mathrm{conv}_{[\iunu(x_i-),\iunu(x_i+)]} f \Big)^{-1} \big( \frac{x}{t} \big) & \iunu(x_i-) < \iunu(x_i+), \\
\Big( \frac{d}{du} \mathrm{conc}_{[\iunu(x_i+),\iunu(x_i-)]} f \Big)^{-1} \big( \frac{x}{t} \big) & \iunu(x_i-) > \iunu(x_i+), \\
\end{cases}
\]
where if $\varsigma : I \to \R$ is increasing and bounded we denote by $\varsigma^{-1}$ its pseudoinverse. It is easy to see that the solution $\tilde u_i(t,x)$ assumes only values in the grid $2^{-\nu} \N$, and that its $L^\infty$-norm is decreasing.

The solution $u^\nu$ for small time is then constructed by piecing together the functions $\unu(t,x)=\tilde u_i(t,x-x_i)$. The solution can be prolonged up to a first time $t_1>0$ when two or more \emph{wavefronts} collide: one then solves again the corresponding Riemann problem and the solution can be thus continued up to a time $t_2>t_1$ where there are wavefront collisions again. Between two collisions, the jump discontinuities propagate with a constant speed, and we will refer to them as \emph{wavefronts}. We say the wavefront located at $x$ is \emph{positive} if the sign of the jump in $x$ is positive and the wavefront is \emph{negative} in the other case.
%

In \cite{Daf-poly} (see also Section 6.1 in \cite{Bre}) it is shown that this procedure can be continued for all time: in fact, at each interaction either the total variation of $\unu$ is decreasing of at least $2^{1-\nu}$ (\emph{cancellation}) or the number of wavefronts decreases by $1$ (\emph{interaction}).

A standard perturbation technique on the speed of the wavefronts allows to assume that each interaction or cancellation point involves only 2 incoming wavefronts and at each time $t$ no more than one collision occurs: let $\{t_j\}$ (with $t_0 = 0$ for convention), be the collision times, and let $x_j$ be the (unique) collision point at time $t_j$. W.l.o.g. we will assume that $u^\nu(t)$ (or $u(t)$, the real solution) is right continuous in $x$ for all $t \geq 0$.

The convergence of $u^\nu(t)$ to $u(t)$ in $L^1_\mathrm{loc}$ for all $t \in \R$ is a consequence of the compactness of BV and the uniqueness of the entropy solution $u$. W.l.o.g. in the following we will suppose that $0 \leq u(t), \unu(t) \leq M$ and that the support of $u,\unu$ is the set $\{|x| \leq C + \Lambda t\}$, $C$ large constant, because of the finite speed of propagation
\begin{equation}
\label{E_Lambda_def}
\Lambda \leq \sup_{|u| \leq \|u_0\|_\infty} |f'(u)|.
\end{equation}

%
%
%
%

\subsection{Wave representation}
\label{s:WR}


In \cite{BM} the authors introduce the \emph{wave representation} of an approximate solution $u^\nu$: this is the unique triple of functions $(\mathtt X^\nu,\mathtt u^\nu,\mathtt a^\nu)$,
%
\begin{equation*}
\begin{array}{rclll}
\mathtt X^\nu &:& \rp \times (0,\tv{\iunu}] \supset E^\nu \to \R, & & \text{the \emph{position of the wave $s$}}, \\
\mathtt u^\nu &:& (0,\tv{\iunu}] \to \R, & & \text{the \emph{value of the wave $s$}}, \\
\mathtt a^\nu &:& \rp \times (0,\tv{\iunu}] \to\{-1,0,1\}, & & \text{the \emph{signed existence interval of the wave $s$}},
\end{array}
\end{equation*}
satisfying the following conditions:
\begin{enumerate}
\item \label{Cond_1} the function $\mathtt a^\nu$ is of the form
\begin{equation*}
\mathtt a^\nu(t,s) = \mathcal S^\nu(s) \chi_{[0,\mathtt T^\nu(s))}(t)
\end{equation*}
for some functions
\begin{equation*}
\begin{array}{rclll}
\mathcal S^\nu &:& (0,\tv{\iunu}] \to \{-1,1\}, & & \text{the \emph{sign of the wave $s$}}, \\
\mathtt T^\nu &:& (0,\tv{\iunu}] \to \R^+, & & \text{the \emph{time of existence of the wave $s$}};
\end{array}
\end{equation*}
\item the set $E^\nu$ is given by
\begin{equation*}
E^\nu = \big\{ (t,s) : \mathtt T^\nu(s) \geq t \big\};
\end{equation*}
\item \label{Cond_2} $s \mapsto \mathtt X^\nu(t,s)$ is increasing for all $t$, $t \mapsto \mathtt X^\nu(t,s)$ is Lipschitz for all $s$, and 
\begin{enumerate}
\item $D_x \unu(t) = \mathtt X^\nu(t,\cdot)_\sharp \Big( a^\nu(t,\cdot)\L^1\mres (0,\tv{\iunu}] \Big)$, i.e. for all $t \geq 0$, $\phi \in C^1(\R)$
\begin{equation*}
- \int_{\R} \unu(t,x) D_x\phi(x)dx = \int_0^{\tv{\iunu}} \phi(X^\nu(t,s))a^\nu(t,s)ds,
\end{equation*}
\item $|D_x \unu|(t) = \mathtt X^\nu(t,\cdot)_\sharp \Big( |a^\nu(t,\cdot)| \L^1\mres (0,\tv{\iunu}] \Big)$;
\end{enumerate}
\item \label{Cond_3} the value $\mathtt u^\nu$ satisfies for all $t < \mathtt T(s)$
\[
\begin{split}
\mathtt u^\nu(s) =&~ D_x \unu(t)(-\infty,x) + \int_{\{s' < s : \mathtt X^\nu(t,s') = \mathtt X^\nu(t,s)\}} a^\nu(t,s') ds' \\
=&~ \unu(t,\mathtt X^\nu(t,s-)) + \int_{\{s' < s : \mathtt X^\nu(t,s') = \mathtt X^\nu(t,s)\}} |a^\nu(t,s')| ds'.
\end{split}
\]
In particular $s \mapsto \mathtt u^\nu(s)$ is a $1$-Lipschitz function of $s$ satisfying 
\begin{equation*}
\frac{d}{ds} \mathtt u^\nu(s) = \mathcal S^\nu(s).
\end{equation*}
\end{enumerate}
The uniqueness is obtained by assuming these functions to be left continuous in space (a choice in accordance with $\unu(t)$ right continuous) and right continuous in time.

We give a sketch of the construction of the wave representation for completeness: an example of the wave representation for a simple wave pattern is given in Figure \ref{f:wave_curve_approx}.

\begin{figure}
        \centering
        \begin{subfigure}[b]{0.40\textwidth}
                \def\svgwidth{180pt}
                 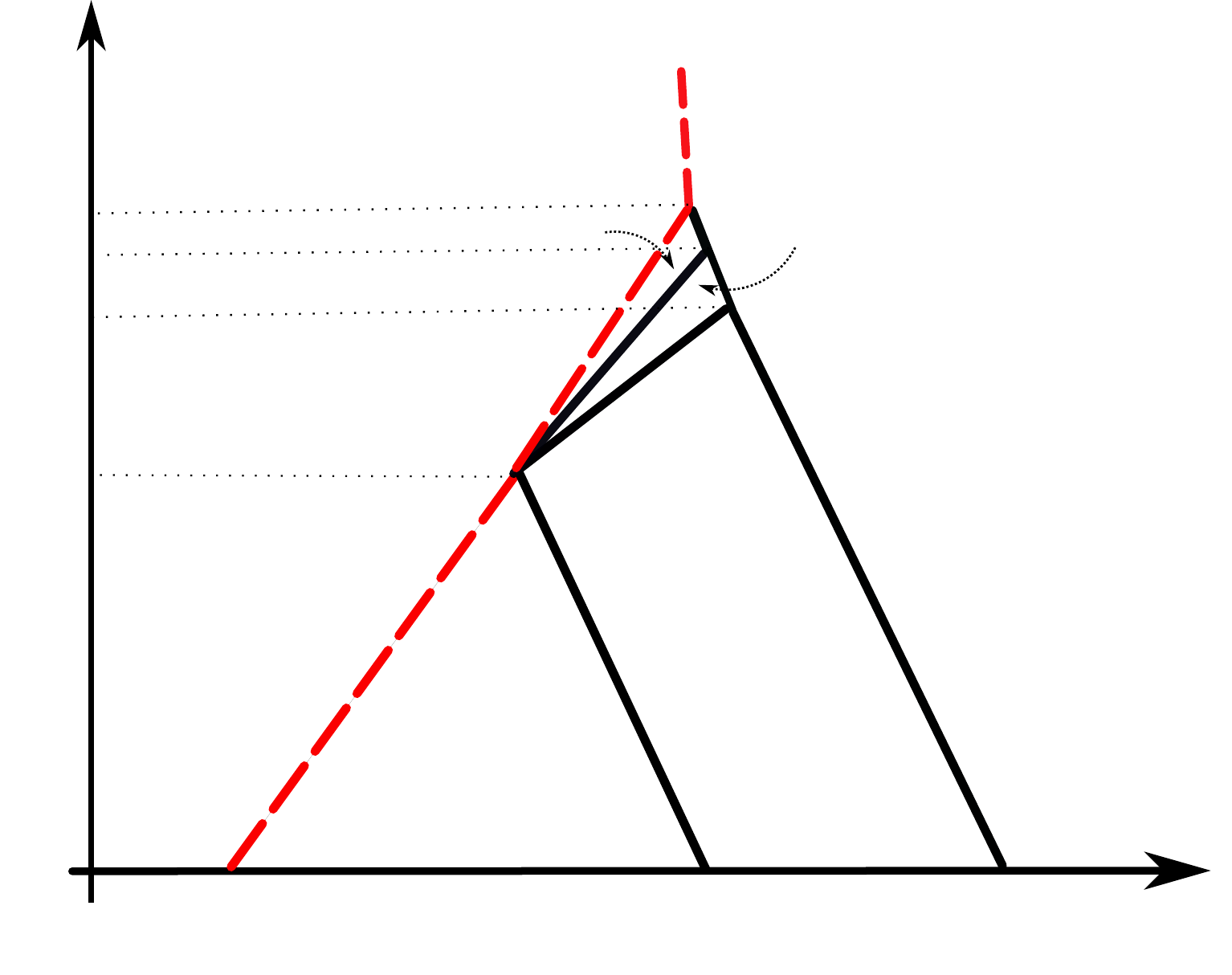
                \caption{wave curves $s\in ]0,1]$.}
                \label{f:wave_curve_survive}
        \end{subfigure}%
        \qquad\qquad 
        \begin{subfigure}[b]{0.40\textwidth}
                \def\svgwidth{180pt}
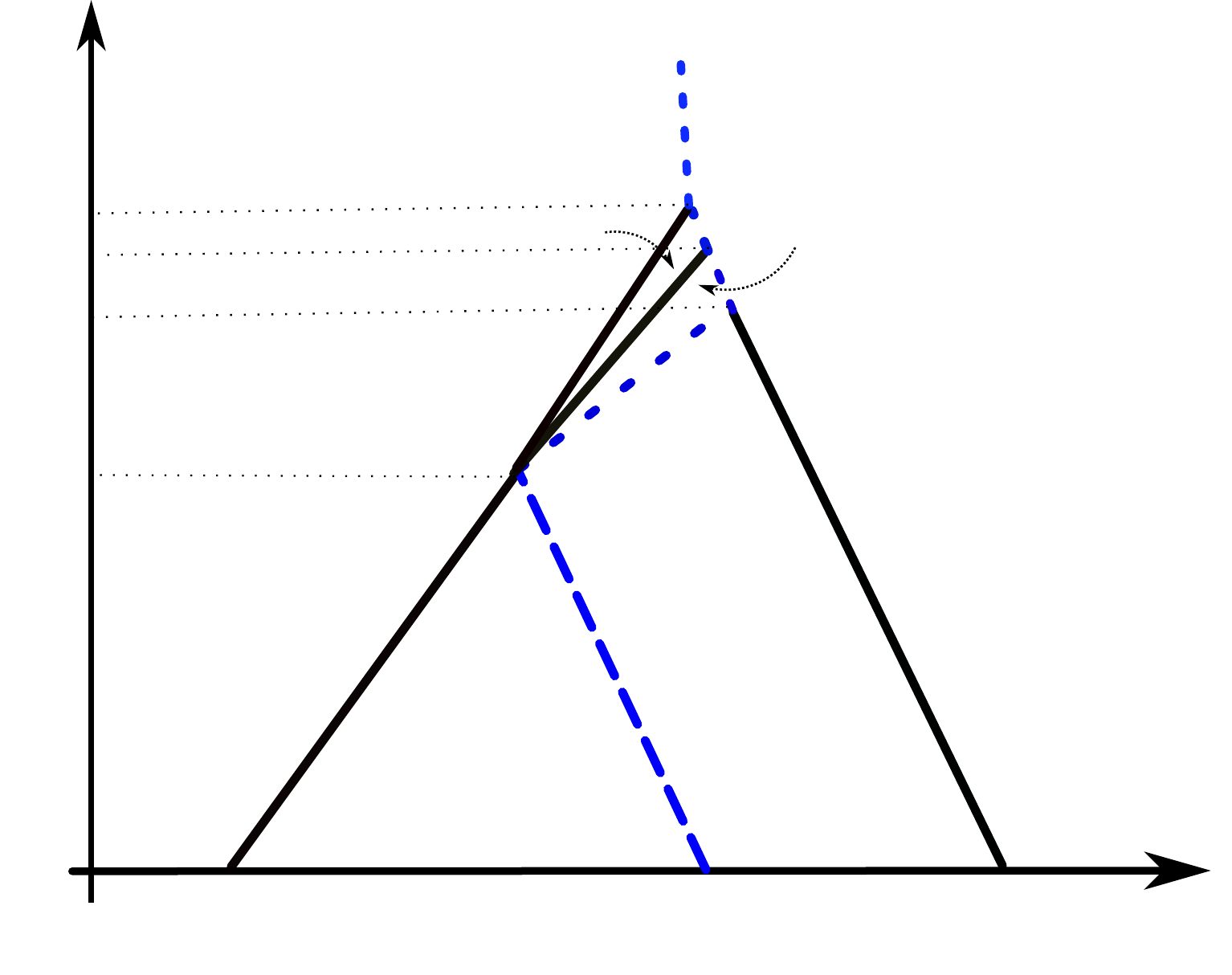            
    \caption{wave curves $s\in ]5,6]$.}
                \label{f:wave_curve_cancel}
        \end{subfigure}
               \caption{Example of the construction of wave curves}\label{f:wave_curve_approx}
\end{figure}

\begin{proposition}
\label{P_uniq_wave_repr}
For any compactly supported approximate solution $\unu(t)$ there exists a unique left continuous wave representation $(\mathtt X^\nu,\mathtt u^\nu,\mathtt a^\nu)$ satisfying conditions \eqref{Cond_1}, \eqref{Cond_2}, \eqref{Cond_3} above.
\end{proposition}

\begin{proof}
%
At $t=0$ it fairly easy to see that $\mathtt X^\nu(t=0)$, $\mathcal S^\nu$ are uniquely defined and
\begin{equation*}
\mathtt u^\nu(s) = \int_0^s \mathcal S^\nu(s') ds'.
\end{equation*}
In the interval $[0,t_1)$ where no interaction occurs, define
\begin{equation}
\label{E_tt_X_nu_sped_def}
\mathtt X^\nu(t,s) := \mathtt X^\nu(0,s) + \sigma^\nu(0,s) t,
\end{equation}
where $\sigma^\nu(0,s)$ is the speed obtained by solving the Riemann problem in the starting location of $s$, i.e.
\begin{equation}
\label{E_speed_RP}
\sigma^\nu(0,s) =
\begin{cases}
\Big( \frac{d}{du} \mathrm{conv}_{(\mathtt X^\nu)^{-1}(0,\mathtt X^\nu(0,s))} f \Big) \big( \mathtt u^\nu(s)-2^{-\nu},\mathtt u^\nu(s) \big) & \mathcal S^\nu(s) > 0, \\
\Big( \frac{d}{du} \mathrm{conc}_{(\mathtt X^\nu)^{-1}(0,\mathtt X^\nu(0,s))} f \Big) \big( \mathtt u^\nu(s),\mathtt u^\nu(s) + 2^{-\nu} \big) & \mathcal S^\nu(s) < 0. \\
\end{cases}
\end{equation}

In order to obtain the representation for every $t \geq 0$, it is thus enough to consider a Riemann problem for two colliding wavefronts at $\bar t$.

If the collision is among wavefronts of the same sign, then the only variation is the speed $\sigma$, and the same formula \eqref{E_tt_X_nu_sped_def} (replacing $\sigma(\bar t-,s)$ with the new speed of the wave exiting the collision point given by \eqref{E_speed_RP}) yields the continuation of $\mathtt X^\nu$ up to the next collision.

If the collision is a cancellation, by the assumption of binary collisions and that $\mathtt u^\nu$ is constant in time, there is a unique interval of waves such that $\mathtt T^\nu(s) = \bar t$. For the other waves the same continuation as in \eqref{E_tt_X_nu_sped_def} (changing the speed of the surviving waves $s$ by solving the Riemann problem at $\bar t$ with the speed given by \eqref{E_speed_RP}) allows to construct the triple till the next collision time.

It is fairly easy to show the validity of the conditions \eqref{Cond_1}, \eqref{Cond_2}, \eqref{Cond_3} after an interaction or a cancellation. The uniqueness follows from the fact that each Riemann problem at $\bar t$ yields a unique set where $\mathtt T^\nu =  \bar t$ and a unique trajectory of the surviving waves.
\end{proof}

\begin{remark}
\label{R_speed_not_exact}
We note that the slight modification of the speed of wavefronts needed in order to have binary collision implies a modification of \eqref{E_tt_X_nu_sped_def}. Since this modification can be chosen as small as we want, we will assume that the speed \eqref{E_tt_X_nu_sped_def} is the actual speed.
\end{remark}

Notice that as a corollary of the above construction it follows that outside the collision times $\{t_j\}_j$ it holds
\begin{equation*}
\sigma^\nu(t,s) := \frac{d}{dt} \mathtt X^\nu(t,s) =
\begin{cases}
\Big( \frac{d}{du} \mathrm{conv}_{(\mathtt X^\nu)^{-1}(t,\mathtt X^\nu(t,s))} f^\nu \Big) \big( \mathtt u^\nu(s)-\nu,\mathtt u^\nu(s) \big) & \mathcal S(s) > 0, \\
\Big( \frac{d}{du} \mathrm{conc}_{(\mathtt X^\nu)^{-1}(t,\mathtt X^\nu(t,s))} f^\nu \Big) \big( \mathtt u^\nu(s),\mathtt u^\nu(s) + \nu \big) & \mathcal S(s) < 0. \\
\end{cases}
\end{equation*}
The function $\sigma(t,s)$ will be called the \emph{speed of the wave $s$}. By assuming it to be $L^1$-right continuous in $t$ and left continuous in $s$ (in accordance with the proof of Proposition \ref{P_uniq_wave_repr}), $\sigma$ is pointwise defined.

Setting $\mathtt a^\nu = 0$ for $\tv{\iunu} < s \leq \tv{\iu}$, we can moreover assume that the interval of integration is always equal to
\begin{equation*}
J := (0,\tv{\iu}].
\end{equation*}
In the following, when not explicitly stated, the interval of integration w.r.t. $s$ will thus be $J$.

A direct computation yields the following

\begin{corollary}
\label{C_u_t_def_1}
It holds
\begin{equation*}
D_t \unu(t) = \mathtt X^\nu(t)_\sharp \big( - \sigma^\nu(t,s) \mathtt a(t,s) \mathcal L^1 \mres J \big) = \mathtt X^\nu(t)_\sharp \bigg( - \frac{d}{dt} \mathtt X^\nu(t) \mathtt a(t,s) \mathcal L^1 \mres J \bigg).
\end{equation*}
\end{corollary}

\begin{proof}
We need only to compute the above term for a Riemann problem: by inspection we have in a collision point $(t_j,x_j)$ and for positive waves
\begin{equation*}
\begin{split}
\int_{(\mathtt X^\nu)^{-1}(t_j,x_j)} \sigma^\nu(t,s) \mathtt a(t,s) ds =&~ \int_{u^\nu(t_j,x_j-)}^{u^\nu(t_j,x_j)} \bigg[ \frac{d}{du} \mathrm{conv}_{(\mathtt X^\nu)^{-1}(t,\mathtt X^\nu(t,s))} f^\nu(u) \bigg] du \\
=&~ f^\nu \big( u^\nu(t_j,x_j) \big) - f^\nu \big( u^\nu(t_j,x_j-) \big),
\end{split}
\end{equation*}
which is exactly the mass of the Dirac delta of $D_x f^\nu(u^\nu(t_j))$ at $x_j$.
\end{proof}

Since for $t \notin \{t_j\}_j$ it holds
\begin{equation*}
\big| \sigma(t,s) - \sigma(t,s') \big| = \bigg| \frac{d}{du} f(u) - \frac{d}{du} f(u') \bigg| \leq \|D^2 f\|_\infty |u - u'|
\end{equation*}
for particular values
\begin{equation*}
u \in (\mathtt X^\nu)^{-1}(t,\mathtt X^\nu(t,s)), \qquad u' \in (\mathtt X^\nu)^{-1}(t,\mathtt X^\nu(t,s')),
\end{equation*}
we conclude also that

\begin{lemma}
\label{L_bv_sigma_x}
Outside collision times, the function $s \mapsto \sigma(t,s)$ has Total Variation bounded by
\begin{equation*}
\tv{\sigma(t)} \leq \|D^2 f\|_\infty \tv{\iunu}.
\end{equation*}
\end{lemma}

%

\begin{remark}
\label{R_dif_scalar}
Observe that in \cite{BM,BM-2} the variable $s$ takes values in $[0,u^\nu_0/\nu] \cap \N$ for the wavefront tracking approximation, while here we let it to assume values in a fixed interval. It is fairly easy to see that our definition is the natural extension obtained by taking $\mathtt X^\nu$ constant on intervals of the form $2^{-\nu} (k,k+1]$.
\end{remark}

\subsection{Uniform bound estimates on variation of speed}
\label{s:BM}

We next present a short proof of the following theorem: a refined estimate of this statement is given in \cite{BM}, where the interested reader can find a discussion about the importance of this result and more technical details about the proof.

\begin{theorem}
\label{T_quadratic}
The following quadratic type estimate holds:
\begin{equation}
\label{E_second_der_sig}
\int_J \tv{\sigma^\nu(s);[0,\mathtt T^\nu(s))} ds \leq 3 \|D^2 f\|_\infty (\tv{\unu})^2.
\end{equation}
\end{theorem}

\begin{proof}
The proof is based on the fact that, using the wave representation, we can distinguish among \emph{pair of waves which have already interacted $\mathcal I^\nu(t)$},
\begin{equation*}
\mathcal I^\nu(t) := \Big\{ (s,s') : s < s' \ \text{and} \ \exists\, t' \leq t \ \big( \mathtt X^\nu(t',s) = \mathtt X^\nu(t',s') \big) \Big\},
\end{equation*}
and \emph{pair of waves which have not yet interacted $\mathcal N^\nu(t)$},
\begin{equation*}
\mathcal N^\nu(t) := \Big\{ (s,s') : s < s' \ \text{and} \ \forall\, t' \leq t \ \big( \mathtt X^\nu(t',s) < \mathtt X^\nu(t',s') \big) \Big\}.
\end{equation*}
Clearly
\begin{equation*}
\mathcal I^\nu(t) \cup \mathcal N^\nu(t) = \big\{ (s,s') : \mathtt a^\nu(t,s), \mathtt a^\nu(t,s') \not= 0 \big\} = \big\{ s : \mathtt \mathtt T^\nu(s) > t \big\} \times \big\{ s : \mathtt \mathtt T^\nu(s) > t \big\}.
\end{equation*}

{\it Step 0.} If $\{(t_j,x_j)\}_j$ are the collision points, we can rewrite (\ref{E_second_der_sig}) as
\begin{equation*}
\begin{split}
\int_J \tv{\sigma^\nu(s);[0,\mathtt T^\nu(s))} ds =&~ \sum_{t_j} \int_{(\mathtt X^\nu)^{-1}(t_j,x_j)} \big| \mathtt a^\nu(t_j,s) \big| \big| \sigma^\nu(t_j,s) - \sigma^\nu(t_j-,s) \big| ds \\
=&~ \sum_{t_j \ \text{cancellation}} \int_{(\mathtt X^\nu)^{-1}(t_j,x_j)} \big| \mathtt a^\nu(t_j,s) \big| \big| \sigma^\nu(t_j,s) - \sigma^\nu(t_j-,s) \big| ds \\
&~ + \sum_{t_j \ \text{interaction}} \int_{(\mathtt X^\nu)^{-1}(t_j,x_j)} \big| \mathtt a^\nu(t_j,s) \big| \big| \sigma^\nu(t_j,s) - \sigma^\nu(t_j-,s) \big| ds.
\end{split}
\end{equation*}
We will estimate the two terms separately.

{\it Step 1.} If a cancellation occurs at $t_j$, then the standard estimate (see Proposition 2.15 of \cite{BM})
\begin{equation*}
\big\| \conv_{[0,b]} f \mres [0,a] - \conv_{[0,a]} f \big\|_{C^1([0,a])} \leq \|D^2 f\|_\infty (b - a), \qquad 0 \leq a \leq b,
\end{equation*}
yields that the variation in speed of the surviving waves is proportional to the cancellation $\mathcal C^\nu(t_j)$ occurring at $t_j$,
\begin{equation*}
\mathcal C^\nu(t_j) := \tv{\unu(t_j-)} - \tv{\unu(t_j)},
\end{equation*}
and thus 
\begin{equation*}
\begin{split}
\sum_{t_j \ \text{cancellation}}& \int_{(\mathtt X^\nu)^{-1}(t_j,x_j)} \big| \mathtt a^\nu(t_j,s) \big| \big| \sigma^\nu(t_j,s) - \sigma^\nu(t_j-,s) \big| ds \\
\leq&~ \|D^2 f\|_\infty \sum_{t_j \ \text{cancellation}} \bigg( \int_{(\mathtt X^\nu)^{-1}(t_j,x_j)} \big| \mathtt a^\nu(t_j,s) \big| ds \bigg) \Big( \tv{\unu(t_j-)} - \tv{\unu(t_j)} \Big) \\
\leq&~ \|D^2 f\|_\infty \sum_{t_j \ \text{cancellation}} \tv{\unu(t_j)} \Big( \tv{\unu(t_j-)} - \tv{\unu(t_j)} \Big) \\
\leq&~ \|D^2 f\|_\infty \tv{\iunu}^2.
\end{split}
\end{equation*}

{\it Step 2.} Define the functional
\begin{equation*}
\mathfrak Q(t) := \mathcal L^2(\mathcal N^\nu(t)).
\end{equation*}
It is easy to see that this functional is decreasing at any collision. 

A simple computation (based on the Rankine-Hugoniot formula of the speed on discontinuities) shows that at each interaction between the wavefronts $w,w'$, with $w$ on the left of $w'$, it holds
\begin{align*}
\int_{(\mathtt X^\nu)^{-1}(t_j,x_j)} \big| \mathtt a^\nu(t_j,s) \big| \big| \sigma^\nu(t_j,s) - \sigma^\nu(t_j-,s) \big| ds =&~ |w| \big( \sigma(w) - \sigma(w+w') \big) + |w'| \big( \sigma(w+w') - \sigma(w') \big) \\
=&~ 2 \frac{(\sigma(w) - \sigma(w')) |w||w'|}{|w+w'|},
\end{align*}
where for shortness we have used the notation $\sigma(w)$ for the speed of the wavefront $w$ and $|w|$ for its strength.

W.l.o.g. in the rest of the proof we will assume $w,w'$ positive, the other case being analogous.

{\it Step 3.} Let $(t_j,x_j)$ be an interaction point of the wavefronts $w,w'$, with $w$ on the left of $w'$: split these wavefronts into the regions
\begin{align*}
\mathcal L_1 :=&~ \Big\{ s \in w : s \ \text{has not interacted with any wave} \ s' \in w' \Big\}, \\
\mathcal L_2 :=&~ \Big\{ s \in w : s \ \text{has already interacted with some wave} \ s' \in w' \Big\}, \\
\mathcal R_1 :=&~ \Big\{ s' \in w' : s \ \text{has already interacted with some wave} \ s \in w \Big\}, \\
\mathcal R_2 :=&~ \Big\{ s' \in w' : s \ \text{has not interacted with any wave} \ s \in w \Big\}.
\end{align*}
A simple argument shows that if $s,s'$ have already interacted then every wave $p$ such that $s \leq p \leq s'$ has already interacted with $s$ and $s'$, so that all the above regions are of the form
\begin{equation*}
\mathcal L_2 = (\ell_1,\ell_0] \cap \big\{ \mathtt a(t_j,s) \not= 0 \big\}, \quad \mathcal R_1 = (\ell_0=r_0,r_1] \cap \big\{ \mathtt a(t_j,s) \not= 0 \big\}.
\end{equation*}

{\it Step 4.} The key observation is now the following:
\begin{equation*}
f(\ell_0 = r_0) = \conv_{[\ell_1,r_1]} f(\ell_0=r_0).
\end{equation*}
For the proof of this fact, we refer to \cite{BM} (see Propositions 2.12 and 3.21, in particular), but one can deduce it from the observation that $\ell_0$ have been split by a previous Riemann problem from the waves $(r_0,r_1]$, and the same can be said for $r_0+\delta$ and $(\ell_1,\ell_0]$ for all $\delta > 0$. \\
Using the above fact and observing that by elementary computations on the graph of the function $f$ one obtains
\begin{equation*}
\frac{(\sigma(w) - \sigma(w'))|w||w'|}{|w+w'|} = f(\ell_0) - \conv_{w \cup w'} f(\ell_0),
\end{equation*}
a simple geometric argument shows that
\begin{align*}
f(\ell_0) - \conv_{w \cup w'} f(\ell_0) \leq&~ \|D^2 f\|_\infty \big( |\mathcal L_1| |w'| + |\mathcal R_2| |w| \big) \\
=&~ \|D^2 f\|_\infty \big( \mathfrak Q(t_j-) - \mathfrak Q(t_j) \big).
\end{align*}

{\it Step 5.} Adding up these contributions and noticing that $\mathfrak Q(0) \leq \tv{\iunu}^2$ we obtain
\begin{equation*}
\begin{split}
\sum_{t_j \ \text{interaction}}& \int_{(\mathtt X^\nu)^{-1}(t_j,x_j)} \big| \mathtt a^\nu(t_j,s) \big| \big| \sigma^\nu(t_j,s) - \sigma^\nu(t_j-,s) \big| ds \\
\leq&~ 2 \sum_{\substack{t_j \ \text{interaction} \\ \text{involving} \ w_j,w_j'}} \frac{(\sigma(w_j) - \sigma(w'_j))|w_j||w'_j|}{|w_j + w'_j|} \\
\leq&~ 2 \sum_{t_j}  \|D^2 f\|_\infty \big( \mathfrak Q(t_j-) - \mathfrak Q(t_j) \big) \\
\leq&~ 2 \|D^2 f\|_\infty \tv{\iunu}^2,
\end{split}
\end{equation*}
which concludes the proof.
\end{proof}

By using the estimate of Lemma \ref{L_bv_sigma_x} we thus conclude that

\begin{proposition}
\label{P_BV_sigma}
The function $\sigma^\nu$ is BV in the region $\{s \in J, 0 \leq t < \mathtt \mathtt T^\nu(s) \}$.
\end{proposition}

\subsection{Estimates on the level sets of the front tracking approximations}
\label{s:elsa}

Since we are considering a right continuous approximate wavefront solution $u^\nu(t)$ with compact support and taking values in $[0,M]$, it follows that for all $w \in [0,M]$ the level set $\{u^\nu > w\}$ is bounded by finitely many Lipschitz curves $\gamma^\nu_{j,w}$, $1 \leq j \leq N^\nu_w$.

%

More precisely, we can take the curves $t \mapsto \gamma^\nu_{j,w}(t)$ piecewise linear, with Lipschitz constant $\Lambda$ independent of $\nu$ and disjoint up to the time $\mathtt T^\nu_{j,w}$. 

It is clear 
that the number $N^\nu_w$ of curves is bounded by the initial datum through coarea formula,
\begin{equation*}
\int_0^M N^\nu_{w} dw = \tv{\iunu},
\end{equation*}
and by Lemma \ref{l:approxbv} we obtain the uniform bound
\begin{equation*}
N^\nu_{w} \leq N_w := \mathcal H^0(\partial \{u>w\}) \in L^1((0,M)).
\end{equation*}

In particular, up to a $\mathcal L^1$-negligible set in $[0,M]$, we have

\begin{lemma}
\label{L_comap_A_nu_iw}
The sets $\{u^\nu > w\}$ are precompact w.r.t. the Hausdorff distance up to a $\mathcal L^1$-negligible set of $w$, and every limit  $A_w$ is a set bounded by the graph of finitely many Lipschitz curves $\gamma_{j,w} : [0,T_{j,w}] \to \R$, $1 \leq j \leq N_w$.
%
\end{lemma}

\begin{proof}
First of all one notices that by monotone convergence definitively $N^\nu_{w} = N_w = \mathcal H^0(\{u>w\}$ up to a $\mathcal L^1$-negligible set of values $w \in [0,M]$. The compactness of the curves $\gamma_{j,w}$, $j = 1,\dots,N_w$, is thus immediate (being uniformly Lipschitz), and the conclusion follows. 
\end{proof}

It is easy to see that the curves $\gamma_{j,w}$, $j=0,\dots,N_w$, are ordered:
\begin{equation*}
\gamma_{j,w}(t) \leq \gamma_{j+1,w}(t),
\end{equation*}
because the same holds for the approximating family $\gamma^\nu_{j,w}$, and moreover that if $n_x$ is the $x$ component of the outer normal vector in \eqref{e:cfss} then
\begin{equation*}
\mathrm{sign}(n_x) = (-1)^{1+j}.
\end{equation*}

For the wavefront approximation $\unu$ one can construct a map $\mathtt U$ between
\begin{equation*}
\mathrm L^\nu := \big\{ (j,w) : 1 \leq j \leq N^\nu_w \big\} \quad \text{and} \quad \mathrm{Graph}\,\mathcal S^\nu.
\end{equation*}
Define in fact
\begin{equation*}
s(j,w) = \tv{\iunu,(-\infty,\gamma^\nu_{j,\nu}(0))} + \big| w - u(\gamma^\nu_{j,w}(0)-) \big|, \quad \mathcal S^\nu(s) = (-1)^j.
\end{equation*}
If we equip the sets with the measures
\begin{equation*}
\mu^\nu := \sum_j \mathcal L^1 \otimes \delta_j, \qquad \omega^\nu := (\id,\mathcal S)_\sharp \mathcal L^1 \mres [0,\tv{u_0^\nu}],
\end{equation*}
respectively, then $\mathtt U$ is a measure preserving bijection out of finitely values $\{0\} \cup \N 2^{-\nu} \cap [0,M]$, whose inverse is
\begin{equation}
\label{E_inverse_sS_wi}
w = \int_0^s \mathcal S(s') ds', \quad j = \sharp \bigg\{ s' < s :  \int_0^{s'} \mathcal S^\nu(s'') ds'' =  \int_0^s \mathcal S^\nu(s'') ds'' \bigg\}.
\end{equation}
The proof of this fact is elementary due to the piecewise constant structure of $u_0^\nu$.

\section{Level sets of the entropy solution}
\label{S_level_1}

In this section we will extend the properties of the level sets for the approximate solutions $\unu(t)$ to the BV solution $u(t)$, showing that there is an $L^1$-representative of $u(t)$ (right continuous in $t$ w.r.t. the $L^1$-norm and right continuous in $x$) enjoying a particularly nice level set structure. As a consequence we will obtain a wave representation for the entropy solution $u(t)$.

Recall that we are restricting our analysis to some compact set $\bar \Omega$, and $u \in [0,M]$ has compact support contained in $\bar \Omega$.


\subsection{Convergence of level sets of front tracking approximations}
\label{s:clsfta}

First, we give an easy lemma about the convergence of the level sets for the approximate solutions $u^\nu$: this lemma is trivial, we give the proof for completeness and only in the case we are interested in.

\begin{lemma}
\label{l:level_set_stable}
Suppose that $v_n$ converges to $v$ as $n \to \infty$ in $L^1(\Omega)$. Then 
\begin{equation}
\label{E_level_set_conv}
\{ v_{n} > w \} \underset{n}{\longrightarrow} \{ v > w \} \quad \text{in} \ L^1(\Omega) 
\end{equation}
for all $w \in \R$ up to a countable set, and viceversa if the convergence \eqref{E_level_set_conv} holds for $\mathcal L^1$-a.e. $w \in [0,M]$ then $v_n \to v$ in $L^1(\Omega)$.

Moreover the values $w$ for which \eqref{E_level_set_conv} does not holds are the ones such that $\{v = w\}$ has positive $\mathcal L^2$-measure.
\end{lemma}

\begin{proof}
Define
\begin{equation*}
\Delta_n(w) := \Big\{ (t,x) \in \Omega : v_n(t,x) > w \geq v(t,x) \ \text{or } v(t,x) > w \geq v_n(t,x) \Big\}, 
\end{equation*}
so that it holds
\[
|v_n-v|(t,x) = \intinf \chi_{\Delta_n(w)}(t,x) dw.
\]
Then by Fubini Theorem,
\begin{equation}
\label{E_Fubini_level_set}
\begin{split}
\int_\Omega|v_n-v|(t,x)dtdx =&~ \int_\Omega\left[\intinf \chi_{\Delta_n(w)}(t,x)dw\right]dtdx \\
=&~ \intinf \left[\int_\Omega\chi_{\Delta_n(w)}(t,x)dtdx\right]dw.
\end{split}
\end{equation}
This yields that, up to a subsequence $\{v_{n'}\}\subset\{v_n\}$, for $\mathcal L^1$-a.e. $w\in \R$,
\[
\int_\Omega\chi_{\Delta_{n'}(w)}(t,x)dtdx\longrightarrow 0 \quad \tas n'\to \infty,
\]
which gives the convergence of level sets in $L^1$-norm up to subsequences.

It is immediate to see that the validity of \eqref{E_level_set_conv} for $\mathcal L^1$-a.e. $w$ implies the $v_n \to v$ in $L^1$ because of \eqref{E_Fubini_level_set} and the finite measure of $\Omega$.

Assume now that that for a fixed $\bar w$ and for some subsequence $v_{n''}$ it holds
\begin{equation*}
\mathcal L^2 \Big( \{ v_{n''} > \bar w \} \setminus \{ v > \bar w \} \Big) \geq \epsilon > 0.
\end{equation*}
By again extracting a subsequence (not relabeled), the previous computation shows that the convergence \eqref{E_level_set_conv} holds $\mathcal L^1$-a.e. $w$, and thus 
\begin{equation*}
\mathcal L^2 \Big( \{ v > w \} \setminus \{v > \bar w\} \Big) = \lim_{n''} \mathcal L^2 \Big( \{ v_{n''} > w \} \setminus \{v > \bar w\} \Big) \geq \epsilon.
\end{equation*}
The monotonicity of the level set gives finally
\begin{equation*}
\mathcal L^2(\{v = \bar w\}) \geq \mathcal L^2 \bigg( \bigcap_{w \nearrow \bar w} \{v > w\} \setminus \{v > \bar w\} \bigg) \geq \epsilon > 0.
\end{equation*}
An analogous reasoning can be done for $\{v > \bar w\} \setminus \{v_n > \bar w\}$.

Thus the only level sets for which there can be a subsequence $\{v_n > w\}$ not converging are the level sets of positive Lebesgue measure, which are at most countable.
\end{proof}

We now use the compactness of the level sets given by Lemma \ref{L_comap_A_nu_iw} for the approximate solution $\unu$ in order to prove the following theorem.

\begin{theorem}
\label{T_right_cts_u}
If $u(t)$ is the right-continuous solution to \eqref{e:cauchy}, then up to a $\mathcal L^1$-negligible set $S \in \R$ the set $\{u > w\}$ is bounded by the graph of finitely many Lipschitz curves $\gamma_{j,w} : [0,T_{j,w}] \to \R$, $1 \leq j \leq N_w$, and moreover 
$\gamma_{j,w}(t) \leq \gamma_{j+1,w}(t)$, $j=1,\dots,N_w$ and $0 \leq t \leq \min\{T_{j,w},T_{j+1,w}\}$.
\end{theorem}

\begin{proof}
Lemma \ref{l:level_set_stable} gives that up to a countable set in $[0,M]$ the sets $\{\unu > w\}$ converge in $L^1$ to $\{u > w\}$.

By Lemma \ref{L_comap_A_nu_iw}, it follows that for $\mathcal L^1$-a.e. $w \in [0,M]$ there exists a subsequence $\nu'$ depending on $w$ such that the level set $\{\unu > w\}$ converges in Hausdorff distance to a set $A_w$ satisfying Lemma \ref{L_comap_A_nu_iw}. 

Since the boundary of the sets $A_w$ are union of finitely many uniformly Lipschitz curves, it follows easily that
\begin{equation*}
\{\unu > w\} \to A_w \quad \text{also in} \ L^1(\Omega),
\end{equation*}
for all subsequences, i.e. the $L^1$-limit of $A^\nu_{i,w}$ is independent on the subsequence.
%

Since the curves $\gamma^\nu_{j,w}$, $j=1,\dots,N_w$, are ordered, the same holds for their limits $\gamma_{j,w}$. 
\end{proof}

It is quite natural to reduce the time interval of existence $[0,T_{j,w}]$ of the curve $\gamma_{j,w}$ by roughly speaking assuming that in the interval $[0,T_{j,w})$ the curve is disjoint from all the others. The idea is that for entropy solutions the total variation decreases also at each level set.

Toward this result, we first prove that the number of disjoint curves is decreasing.

\begin{lemma}
\label{L_number_decre}
For $\mathcal L^1$-a.e $w \in [0,M]$ the following holds: if $\gamma_{j,w}(\bar t) = \gamma_{j',w}(\bar t)$, then we can assume that for all $t \geq \bar t$
\begin{equation*}
\gamma_{j,w}(t) = \gamma_{j',w}(t).
\end{equation*}
\end{lemma}


\begin{proof}
%
W.l.o.g. we can assume that $t=0$, and that $w$ is one of the values of the level sets for which Theorem \ref{T_right_cts_u} holds: the set $\partial \{u > w\}$ is the union of finitely many Lipschitz curves $\{\gamma_{j,w}\}_{j=1,\dots,N_w}$, which are uniform limit of the wavefront approximations. In particular the points $\{\gamma_{j,w}(0)\}_j$ are isolated. 

We first notice the following basic fact: if on a Lipschitz curve $\bar \gamma(t)$ it holds
\begin{equation*}
\lim_{x \searrow \bar \gamma(t)} u(t,x) = w,
\end{equation*}
then the function
\begin{equation*}
\tilde u(t,x) =
\begin{cases}
w & x \leq \bar \gamma(t), \\
u(t,x) & x > \bar \gamma(t),
\end{cases}
\end{equation*}
is an entropy solution. The proof is elementary, since on $\bar \gamma$ the function $\tilde u$ has no jumps, and it is entropic in the two open sets $\{x \gtrless \bar \gamma(t) \}$. The same holds if
\begin{equation*}
\lim_{x \nearrow \bar \gamma(t)} u(t,x) = w.
\end{equation*}
In this latter case one verifies that the jump conditions are satisfied on the curve $\bar \gamma(t)$.

Assume by contradiction that two curves $\gamma_{j,w}$, $\gamma_{j+1,w}$ start at $(0,0)$ with $\gamma_{j,w}(t) < \gamma_{j+1,w}(t)$ for $t > 0$. Then the function
\begin{equation*}
\hat u(t,x) :=
\begin{cases}
0 & x < \gamma_{j,w}(t), \\
u(t,x) & \gamma_{j,w}(t) \leq x < \gamma_{j+1,w}(t), \\
0 & x \geq \gamma_{j+1,w}(t),
\end{cases}
\end{equation*}
is an entropy solution with initial datum $\hat u(t=0) = w$, which forces $\hat u(t) = w$ ans $\gamma_{j,w} \equiv \gamma$. This concludes the proof.
\end{proof}

It is now clear that for the level sets $\{u > w\}$ made of finitely many curves we can define the time of existence $\mathtt T_{j,w}$ as follows: if in $(\bar t,\bar x)$ the curves $\gamma_{j,w}$, $j = N_1,\dots,N_2$, join, then set
\begin{equation*}
\mathtt T_{j,w} = \bar t \quad \text{for} \quad j = N_1,\dots,N_1 + 2 \bigg[ \frac{N_2 - N_1 + 1}{2} \bigg] - 1,
\end{equation*}
where $[a]$ is the integer part of $a$. We notice that our choice is not unique: in fact, we decide to prolong the rightmost curve if an odd number of curves meet in a point.

With this choice we obtain the next corollary.

\begin{corollary}
\label{C_wave_repr}
For $\mathcal L^1$-a.e. $w \in [0,M]$ the Lipschitz curves $\gamma_{j,w}$ of Theorem \ref{T_right_cts_u} can be restricted to an interval $[0,\mathtt T_{j,w})$ such that $\gamma_{j,w}(t) \not= \gamma_{j',w}(t)$ for all $0 \leq t \leq \min\{\mathtt T_{j,w},\mathtt T_{j',w}\}$. 
%
\end{corollary}


%
%
%
%
%
%
%
%

\subsection{Wave representation of the entropy solution}
\label{Ss_lagr_repr_real_u}

We now reparametrize the curves $\gamma_{j,w}$ in order to obtain a Lagrangian representation, i.e. a wave representation for the entropy solution. This has to be done at $t=0$, because the parametrization is independent of time for the wave representation as well as for the level set description of Theorem \ref{T_right_cts_u}.

Define the maps
\begin{equation}
\label{E_mathtt_s_def}
\begin{split}
\mathtt s(j,w) :=&~ \sum_{j'} \int \chi_{\gamma_{j',w'}(0) < \gamma_{j,w}(0)}(w') dw' + \sum_{j'} \int_{\gamma_{j',w'}(0) = \gamma_{j,w}(0)} \chi_{(-1)^j(w - w') < 0}(w') dw' \\
=&~ \tv{u_0,(-\infty,\gamma_{j,w})(0)} + \sum_{j'} \int_{\gamma_{j',w'}(0) = \gamma_{j,w}(0)} \chi_{(-1)^j(w - w') < 0}(w') dw' \crcr
=&~ \tv{u_0,(-\infty,\gamma_{j,w})(0)} + \big| w - u_0(\gamma_{j,w}(0)-) \big|,
\end{split}
\end{equation}
and consider
\begin{equation*}
D := \big\{ (j,w) : 1 \leq j \leq N_w \big\}, \quad \mu := \sum_j \delta_j \times \mathcal L^1 \mres D(j).
\end{equation*}
We have used the multifunction notation, i.e. $D(j) = \big\{ w : (j,w) \in D \big\}$. The map $\mathtt s$ has domain $D$ and range $J = (0,\tv{u_0}]$.

\begin{proposition}
\label{P_map_s_jw}
The map $\mathtt s$ is invertible up to a $\mathcal L^1$-negligible subset of $J =( 0,\tv{u_0}]$ and defines a measure preserving bijection.
\end{proposition}

\begin{proof}
First of all, $\mathtt s(j,w) = \mathtt s(j',w')$ if and only if
\begin{equation*}
\begin{split}
& \big| w - u_0(\gamma_{j,w}(0)) \big| + \tv{u_0,(\gamma_{j,w}(0),\gamma_{j',w'}(0))} + \big| w' - u_0(\gamma_{j',w'}(0)-) \big| = 0,
\end{split}
\end{equation*}
assuming for definiteness $\gamma_{j,w}(0) \leq \gamma_{j',w'}(0)$. Hence we conclude that
\begin{equation*}
w = w' \quad \text{and} \quad \tv{u_0,(\gamma_{j,w}(0),\gamma_{j',w}(0))} = 0.
\end{equation*}
Thus either $j = j'$ or $w$ is one of the at most countable set of values for which $\mathcal L^1(\{u = w\}) > 0$. We thus conclude that up to a countable set $\mathtt s$ is injective. Moreover, b

By coarea formula and the definition of $\mathtt s$
\begin{equation*}
\begin{split}
\mathtt s(j',w') - \mathtt s(j,w) =&~ \big| w - u_0(\gamma_{j,w}(0)) \big| + \tv{u_0,(\gamma_{j,w}(0),\gamma_{j',w'}(0))} + \big| w' - u_0(\gamma_{j',w'}(0)-) \big| \crcr
=&~ \mu \bigg( \Big\{ (j'',w''), \gamma_{j,w}(0) < \gamma_{j'',w''}(0) < \gamma_{j',w'}(0) \Big\} \crcr
&~ \qquad \bigcup \Big\{ (j'',w''): \gamma_{j'',w''}(0) = \gamma_{j,w}(0) \ \text{and} \ (-1)^{j}(w - w'') < 0 \Big\} \crcr
&~ \qquad \bigcup \Big\{ (j'',w''): \gamma_{j'',w''}(0) = \gamma_{j',w'}(0) \ \text{and} \ (-1)^{j'}(w' - w'') < 0 \Big\} \bigg) \crcr
=&~ \mu \big( \mathtt s^{-1} \big( \mathtt s(j,w),\mathtt s(j',w') \big) \big).
\end{split}
\end{equation*}
Since
\begin{equation*}
\sup_{j,w} \mathtt s(j,w) = \tv{u_0},
\end{equation*}
the above computation shows that $\mathtt s_\sharp \mu = \mathcal L^1 \mres J$, and thus $\mathtt s$ is invertible $\mathcal L^1$-a.e. and the proof is complete.
\end{proof}

Now we naturally extend the \emph{wave representation} of an entropy solution $u$.

\begin{definition}
\label{D_wave_repres}
A wave representation is a triple of functions $(\mathtt X,\mathtt u,\mathtt a)$,
%
\begin{equation*}
\begin{array}{rclll}
\mathtt X &:& \rp \times (0,\tv{u_0}] \supset E \to \R, & & \text{the \emph{position of the wave $s$}}, \\
\mathtt u &:& (0,\tv{u_0}] \to \R, & & \text{the \emph{value of the wave $s$}}, \\
\mathtt a &:& \rp \times (0,\tv{u_0}] \to \{-1,0,1\}, & & \text{the \emph{signed existence interval of the wave $s$}},
\end{array}
\end{equation*}
satisfying the following conditions:
\begin{enumerate}
\item \label{Cond_1_entrop} the function $\mathtt a$ is of the form
\begin{equation*}
\mathtt a(t,s) = \mathcal S(s) \chi_{[0,\mathtt T(s))}(t)
\end{equation*}
for some functions
\begin{equation*}
\begin{array}{rclll}
\mathcal S &:& (0,\tv{u_0}] \to \{-1,1\}, & & \text{the \emph{sign of the wave $s$}}, \\
\mathtt T &:& (0,\tv{u_0}] \to \R^+, & & \text{the \emph{time of existence of the wave $s$}};
\end{array}
\end{equation*}
\item the set $E$ is given by
\begin{equation*}
E = \big\{ (t,s) : \mathtt T(s) \geq t \big\};
\end{equation*}
\item \label{Cond_2_entrop} $s \mapsto \mathtt X(t,s)$ is increasing for all $t$, $t \mapsto \mathtt X(t,s)$ is Lipschitz for all $s$, and 
\begin{enumerate}
\item $D_x u(t) = \mathtt X(t,\cdot)_\sharp \Big( a(t,\cdot)\L^1 \mres (0,\tv{u_0}] \Big)$, i.e. for all $t \geq 0$, $\phi \in C^1(\R)$
\begin{equation*}
- \int_{\R} u(t,x) D_x\phi(x)dx = \int_0^{\tv{u_0}} \phi(X(t,s)) a(t,s) ds,
\end{equation*}
\item $|D_x u|(t) = \mathtt X(t,\cdot)_\sharp \Big( |a(t,\cdot)| \L^1 \mres (0,\tv{u}] \Big)$;
\end{enumerate}
\item \label{Cond_3_entrop} the value $\mathtt u$ satisfies for all $t < \mathtt T(s)$
\[
\begin{split}
\mathtt u(s) =&~ D_x u(t)(-\infty,x) + \int_{\{s' < s : \mathtt X(t,s') = \mathtt X(t,s)\}} a(t,s') ds' \\
=&~ u(t,\mathtt X(t,s-)) + \int_{\{s' < s : \mathtt X(t,s') = \mathtt X(t,s)\}} |a(t,s')| ds'.
\end{split}
\]
\end{enumerate}
\end{definition}

Once we have constructed the function $\mathtt s$ in \eqref{E_mathtt_s_def}, we define accordingly
\begin{equation*}
\mathtt X(t,s(j,w)) := \gamma_{j,w}(t),
\end{equation*}
\begin{equation*}
\mathcal S(s(j,w)) := (-1)^{j+1},
\end{equation*}
\begin{equation*}
\mathtt T(s(j,w)) := \mathtt T_{j,w}.
\end{equation*}
These maps are defined $\mathcal L^1$-a.e. on $J = [0,\tv{u_0}]$ by the previous proposition. 

\begin{theorem}
\label{T_wave_repr_cont}
The functions
\begin{equation*}
\mathtt X(t,\mathtt s(j,w)) = \gamma_{j,w}(t), \quad \mathtt u(\mathtt s(j,w)) := w, \quad \mathtt a(t,\mathtt s(j,w)) := (-1)^{j+1} \chi_{[0,\mathtt T_{j,w})}(t)
\end{equation*}
provide a wave representation of the entropy solution $u(t)$.
\end{theorem}

\begin{proof}
The fact that $\mathtt X(t,s)$ is increasing in $s$ and Lipschitz in $t$ is immediate from the fact that $(\mathtt s(j,w),\gamma_{j,w}(t))$ is monotone and $t \mapsto \gamma_{j,w}$ is Lipschitz. In order to prove that it is a wave representation, we are left to prove that
\begin{equation}
\label{E_wave_cts_1}
D_x u(t) = \mathtt X(t)_\sharp \big( \mathtt a(t) \mathcal L^1 \big), \quad |D_x u(t)| = \mathtt X(t)_\sharp \big( |\mathtt a(t)| \mathcal L^1 \big),
\end{equation}
and that
\begin{equation*}
\mathtt u(s) = \int_{s' < s} \mathtt a(t,s) ds.
\end{equation*}
This last equation follows immediately from the definition of $\mathtt s(j,w)$ and the fact that the value $w$ is constant on the curve $\gamma_{j,w}$, while the other two follow from coarea formula, the measure preserving property of the map $\mathtt s$ and the fact that the curves $\gamma_{i,w}$ are disjoint for $\mathcal L^1$-a.e. $w \in [0,\tv{u_0}]$.
\end{proof}

In order to make the wave representation defined pointwise, just extend it in order to be left continuous.

To conclude this section we prove two additional estimates on the wave representation, which are trivially verified for the wavefront approximations. The first is an auxiliary result, which is a consequence of the fact that $\mathtt u(s)$ is not time dependent.

\begin{lemma}
\label{L_no_canc_proj}
The following holds:
\begin{equation*}
\mathtt X_\sharp \big( \mathcal S(s) \partial_t \chi_{t < \mathtt T(s)} \big) = 0.
\end{equation*}
\end{lemma}

\begin{proof}
It is immediate to see that
\begin{equation*}
\partial_t \chi_{t \leq \mathtt T(s)}(t,s) = - (\id,\mathtt T)_\sharp \mathcal L^1 \mres J.
\end{equation*}
In fact, for all $\phi \in C^1_c(\R^+ \times J)$
\begin{equation*}
\begin{split}
\int \partial_t \phi \chi_{t \leq \mathtt T(s)} dtds =&~ \int \phi(\mathtt T(s),s) ds = \int \phi(t,s) \big( (\id,\mathtt T)_\sharp \mathcal L^1\big) (dtds).
\end{split}
\end{equation*}
Write the disintegration
\begin{equation*}
(\id,\mathtt T)_\sharp \mathcal L^1 \mres J = \int \mu_t(ds) m(dt).
\end{equation*}

Using the fact that $\partial_t \mathtt u(s) = 0$ one obtains immediately
\begin{equation*}
\begin{split}
0 =&~ \partial_t \int_0^{\bar s} \mathcal S(s) \chi_{t < \mathtt T(s)}(s) ds = - \int_0^{\bar s} \mathcal S(s) \mu_t(ds).
\end{split}
\end{equation*}
%
If $x < x'$ are continuity points of $u$ which are not image points of any jump of $\mathtt X(t)$, then there exists unique point $s < s'$ such that
\begin{equation*}
x = \mathtt X(t,s) , \qquad x' = \mathtt X(t,s').
\end{equation*}
In particular we obtain
\begin{equation*}
\mathtt X(t)_\sharp \big( \mathcal S(t) \mu_t \big)([x,x']) = \int_s^{s'} \mathcal S(s) \mu_t(ds) = 0.
\end{equation*}
This proves the statement.
\end{proof}

Define the \emph{cancellation measure} as
\begin{equation*}
\xi^\mathrm{canc} := \mathtt X_\sharp (\partial_t \chi_{t < \mathtt T}) = \big( \mathtt X \circ (\mathtt T,\id) \big)_\sharp \mathcal L^1. 
\end{equation*}
An immediate corollary is the following.

\begin{corollary}
\label{C_canc_on_jumps}
The cancellation measure is concentrated on the jumps of $u(t)$.
\end{corollary}

\begin{proof}
We just need to observe that from Lemma \ref{L_no_canc_proj} it holds
\begin{equation*}
\mu_t \big( \big\{ s : \mathtt X^{-1}(t,\mathtt X(t,s)) \ \text{singleton} \big\} \big) = 0,
\end{equation*}
and that the points such that
\begin{equation*}
\int_{\mathtt X(t)^{-1}(x)} \chi_{t < \mathtt T(s)} ds > 0
\end{equation*}
are the discontinuities of $u(t)$. 
\end{proof}


The second result shows that the curves $t \mapsto \mathtt X(t,s)$, $t \in [0,\mathtt T(s))$, are characteristic curve of
\begin{equation*}
u_t + f(u)_x = 0.
\end{equation*}
Writing the above equation as
\begin{equation*}
D_t u + \tilde \lambda D_x u = 0,
\end{equation*}
with $\tilde \lambda$ given by Volpert rule,
\begin{equation*}
\tilde \lambda(t,x) :=
\begin{cases}
f'(u(t,x)) & u(t) \ \text{continuous in} \ x, \\
\frac{f(u(t,x-)) - f(u(t,x+))}{u(t,x-) - u(t,x+)} & \text{otherwise},
\end{cases}
\end{equation*}
we obtain
\begin{equation*}
D_t u = \mathtt X_\sharp \Big( - \tilde \lambda(t,\mathtt X(t)) \mathcal S \chi_{t<\mathtt T(s)} \mathcal L^1 \Big).
\end{equation*}

\begin{proposition}
\label{P_cancell_points}
The following holds: for $\mathcal L^1$-a.e. $t \in [0,\mathtt T(s))$ 
\begin{equation*}
\partial_t \mathtt X(t,s) = \tilde \lambda(t,\mathtt X(t,s)).
\end{equation*}
\end{proposition}

\begin{proof}
%
%
%
By writing
\begin{equation*}
\begin{split}
\int \partial_x \phi D_t u dx =&~ \int \partial_t \phi D_x u dt \\
(\text{by} \ \eqref{E_wave_cts_1}) =&~ \int \partial_t \phi(t,X(t,s)) \mathcal S(s) \chi_{t < \mathtt T(s)} dsdt \\
=&~ \int \bigg( \frac{d}{dt} \phi(t,X(t,s)) - \partial_x \phi(t,\mathtt X(t,s)) \frac{d}{dt} \mathtt X(t,s) \bigg) \mathcal S(s) \chi_{t < \mathtt T(s)} dsdt \\
=&~ \int \phi(\mathtt T(s),\mathtt X(\mathtt T(s),s)) \mathcal S(s) ds - \int \partial_x \phi(t,\mathtt X(t,s)) \sigma(t,s) \mathcal S(s) \chi_{t < \mathtt T(s)} ds dt \\
(\text{by Lemma \ref{L_no_canc_proj}}) =&~ - \int \partial_x \phi(t,\mathtt X(t,s)) \sigma(t,s) \mathcal S(s) \chi_{t < \mathtt T(s)} ds dt \\
\end{split}
\end{equation*}
where we have defined
\begin{equation*}
\sigma(t,s) := \frac{d}{dt} \mathtt X(t,s).
\end{equation*}
We thus conclude that
\begin{equation*}
\int \partial_x \phi(t,\mathtt X(t,s)) \big( \sigma(t,s) - \tilde \lambda(t,\mathtt X(t,s)) \big) \mathcal S(s) \chi_{t < \mathtt T(s)} ds dt = 0.
\end{equation*}
Using the arbitrariness of $\phi$ one thus obtains that up to a $\mathcal L^1$-negligible set of times
\begin{equation*}
\mathtt X(t)_\sharp \Big( \big( \sigma(t,s) - \tilde \lambda(t,\mathtt X(t,s)) \big) \mathcal S(s) \chi_{t < \mathtt T(s)} \mathcal L^1 \Big) = 0.
\end{equation*}

We thus conclude that for $\mathcal L^1$-a.e. $s$ such that $\mathtt X^{-1}(t,\mathtt X(t,s))$ is single valued it holds
\begin{equation*}
\sigma(t,s) = \tilde \lambda(t,\mathtt X(t,s)) = f'(u(t,\mathtt X(t,s))).
\end{equation*}
In the intervals of the form $\mathtt X^{-1}(t,x)$ for some $x$ we obtain
\begin{equation*}
\int_{\mathtt X^{-1}(t,x)} \big( \sigma(t,s) - \tilde \lambda(t,\mathtt X(t,s)) \big) \mathcal S(s) \chi_{t < \mathtt T(s)} ds = 0.
\end{equation*}
Since the points of $L^1$-continuity of $t \mapsto \sigma(t,s)$ are dense in $t$, by removing a $\mathcal L^1$-negligible set of $t$ we have that $t \mapsto \sigma(t,s)$ is continuous up to a $\mathcal L^1$-negligible set of $s$.

In these times the only possibility in order to have $s \mapsto \mathtt X(t,s)$ monotone is that
\begin{equation*}
\sigma(t,s) = \tilde \lambda(t,\mathtt X(t,s)).
\end{equation*}
This concludes the proof.
\end{proof}

Finally the convergence of the curves $\gamma_{i,w}(t)$ as shown in the proof of Theorem \ref{T_right_cts_u} gives the following estimate.

\begin{theorem}
\label{T_quadr_contin}
The following holds:
\begin{equation*}
\int_J \tv{\sigma(t,s), [0,\mathtt T(s)) } ds \leq 3\|f''\|_\infty \tv{u_0,\R}^2.
\end{equation*}
\end{theorem}

The measure
\begin{equation*}
\xi^\mathrm{inter} := \mathtt X_\sharp \bigg( \int_J \partial_t \sigma(s) ds \bigg)
\end{equation*}
will be called the \emph{interaction measure}.

\section{Pointwise regularity}
\label{S_pointwise_regul}

This last section concerns the additional regularity which can be deduced from the wave representation.

We recall that a spacelike curve $\R \ni x \mapsto \tau(x) \in \R^+$ is a Lipschitz curve such that $|\tau'| \leq \Lambda^{-1}$, where $\Lambda$ is given by \eqref{E_Lambda_def}. For any $s \in J$ and spacelike curve $\tau$ let $\mathtt t_\tau(s)$ be the unique time such that
\begin{equation*}
\mathtt t_\tau(s) = \tau \big( \mathtt X(\mathtt t_\tau(s),s) \big).
\end{equation*}

Let $\tau$ be a spacelike curve and fix two points $(\tau(x),x)$, $(\tau(x'),x')$, $x \leq x'$. Let
\begin{equation*}
s \in \mathtt X^{-1}(\tau(x),x), \qquad s' \in \mathtt X^{-1}(\tau(x'),x'),
\end{equation*}
with $s \leq s'$.

The first lemma is an immediate consequence of the wave representation.

\begin{lemma}
\label{L_exact_canc}
The following holds: if $\tau' \leq \tau$ is another spacelike curve, then
\begin{equation*}
\begin{split}
\tv{u \circ \tau,(x,x')}& + \big| \mathtt u(s) - (u \circ \tau)(x+) \big| + \big| \mathtt u(s') - (u \circ \tau)(x'-) \big| \\
=&~ \tv{u \circ \tau',(\mathtt X(\mathtt t_{\tau'}(s),s),\mathtt X(\mathtt t_{\tau'}(s'),s')} \\
&~ + \big| \mathtt u(s) - (u \circ \tau')(\mathtt X(\mathtt t_{\tau'}(s),s)+) \big| + \big| \mathtt u(s') - (u \circ \tau')(\mathtt X(\mathtt t_{\tau'}(s'),s')-) \big| \\
&~ - \mathcal L^1 \Big( \Big\{ s'' \in (s,s'), \mathtt t_{\tau'}(s'') < \mathtt T(s'') \leq \mathtt t_{\tau}(s'') \Big\} \Big).
\end{split}
\end{equation*}
\end{lemma}

Note that the last term is equal to the cancellation occurring the the region bounded by the two spacelike curves and the interval $(s,s')$.

\begin{proof}
The proof is immediate from the definition of total variation.
\end{proof}

Observing that $\mathtt u(s)$ is constant along the trajectory $t \mapsto \mathtt X(t,s)$, we obtain the corresponding result for the $L^\infty$-estimate.

\begin{lemma}
\label{L_exact_l_infty_decr}
In the same setting as above, it holds
\begin{equation*}
\begin{split}
\sup_{(x,x')} \Big\{ \big|u \circ \tau - \bar u\big| + \big| \mathtt u(s) - \bar u \big|& + \big| \mathtt u(s') - \bar u \big| \Big\} + \mathcal L^1 \Big( \Big\{ s'' \in (s,s'), \mathtt t_{\tau'}(s'') < \mathtt T(s'') \leq \mathtt t_{\tau}(s'') \Big\} \Big) \\
\geq&~ \sup \Big\{ \big|u \circ \tau' - \bar u \big| + \big| \mathtt u(s) - \bar u \big| + \big| \mathtt u(s') - \bar u \big|, (\mathtt X(\mathtt t_{\tau'}(s),s),\mathtt X(\mathtt t_{\tau'}(s'),s') \Big\}, 
\end{split}
\end{equation*}
for two spacelike curves $\tau' \leq \tau$ and for all $\bar u \in \R$.
\end{lemma}

\begin{proof}
One just needs to observe that
\begin{equation*}
\begin{split}
\sup \Big\{ \big|u \circ \tau' - \bar u \big|& + \big| \mathtt u(s) - \bar u \big| + \big| \mathtt u(s') - \bar u \big|, (\mathtt X(\mathtt t_{\tau'}(s),s),\mathtt X(\mathtt t_{\tau'}(s'),s') \Big\} \\
&~ - \sup_{(x,x')} \Big\{ \big|u \circ \tau - \bar u\big| + \big| \mathtt u(s) - \bar u \big| + \big| \mathtt u(s') - \bar u \big| \Big\} \crcr
\leq&~ \bigg[ \tv{u \circ \tau',(\mathtt X(\mathtt t_{\tau'}(s),s),\mathtt X(\mathtt t_{\tau'}(s'),s')} \\
&~ \qquad \quad + \big| \mathtt u(s) - (u \circ \tau')(\mathtt X(\mathtt t_{\tau'}(s),s)+) \big| + \big| \mathtt u(s') - (u \circ \tau')(\mathtt X(\mathtt t_{\tau'}(s'),s')-) \big| \bigg] \crcr
&~ \quad - \bigg[ \tv{u \circ \tau,(x,x')} + \big| \mathtt u(s) - (u \circ \tau)(x+) \big| + \big| \mathtt u(s') - (u \circ \tau)(x'-) \big| \bigg] \\
=&~ \mathcal L^1 \Big( \Big\{ s'' \in (s,s'), \mathtt t_{\tau'}(s'') < \mathtt T(s'') \leq \mathtt t_{\tau}(s'') \Big\} \Big).
\end{split}
\end{equation*}
\end{proof}

Using the finite speed of propagation $\Lambda$, one obtains the following well known result.

\begin{corollary}
\label{C_triangl_dep_tv}
In the domain of dependence
\begin{equation*}
\triangle_{\bar t,a,b} := \Big\{ (t,x), \bar t \leq t \leq (b - a)/\Lambda, a + \Lambda (t-\bar t) \leq x \leq b - \Lambda (t-\bar t) \Big\}, \qquad a < b,
\end{equation*}
it holds
\begin{equation*}
\tv{u(\bar t),(a,b)} \geq \tv{u(t), \triangle(t)} + \xi^\mathrm{inter} \big( \triangle_{\bar t,a,b} \cap (\bar t,t] \big),
\end{equation*}
\begin{equation*}
\sup_{(a,b)} \big|u(\bar t) - \bar u\big| \leq \sup \big\{ \big|u(t) - \bar u\big|, \triangle(t) \big\} + \xi^\mathrm{inter} \big( \triangle_{\bar t,a,b} \cap (\bar t,t] \big).
\end{equation*}
\end{corollary}

As a first regularity estimate, we show that the set
\begin{equation*}
\Xi := \big\{ (t,x) : u(t,x-) \not= u(t,x+) \big\}
\end{equation*}
is $1$-rectifiable.

\begin{lemma}
\label{L_jump_set}
The set $\Xi$ is $1$-rectifiable.
\end{lemma}

\begin{proof}
The proof is an immediate consequence of the wave representation: one in fact observes that this set is covered by
\begin{equation*}
\Xi' := \big\{ (t,x) : \mathtt X^{-1}(t,x) \ \text{is not a singleton} \big\},
\end{equation*}
and this is covered by the dense set of curves $t \mapsto \mathtt X(t,q)$, $q \in \Q \cap J$, due to the fact that $\mathcal L^1(\mathtt X^{-1}(t,x)) \geq |D_x u|(t,x)$.
%
\end{proof}

We now prove the first regularity result of this section. This result is more or less already known (see \cite{Liu_adm}), but here we can use Corollary \ref{C_canc_on_jumps} to improve slightly the result, showing that the countable set of strong interactions/cancellations are on $\Xi$.

\begin{proposition}
\label{P_cont_out_jumps}
The function $u$ is continuous outside $\Xi$.
\end{proposition}

\begin{proof}
Fix a point $(\bar t,\bar x) \notin \Xi$. By Corollary \ref{C_triangl_dep_tv} one has only to study the case $t < \bar t$.

Assume by contradiction that there exists a sequence of $t_i \nearrow \bar t$, $x_i \to x$ such that
\begin{equation*}
\big| u(t_i,x_i) - u(\bar t,\bar x) \big| > \e.
\end{equation*}
By Lemma \ref{L_exact_l_infty_decr} applied with $\bar u = u(\bar t,\bar x)$, we deduce that in all sufficiently small neighborhoods $O$ of $(\bar t,\bar x)$ a cancellation of size $\xi^\mathrm{canc}(O) > \e/2$ occurs, being $\bar x$ a continuity point of $u(\bar t)$.
%
%
It thus follows that $(\bar t,\bar x)$ is a Dirac-delta for the cancellation measure $\xi^\mathrm{canc}$, which implies by Corollary \ref{C_canc_on_jumps} that $(\bar t,\bar x) \in \Xi$, contradicting the assumption.
\end{proof}

We finally study the regularity across the set $\Xi$, extending the analogous result for genuinely nonlinear scalar conservation laws, see \cite{Bre}. For any point $(\bar t,\bar x)$ consider the set
\begin{equation*}
\mathtt S(\bar t,\bar x) :=
\begin{cases}
\mathtt X^{-1}(\bar t,\bar x) & \text{if single valued}, \\
\text{interior of} \ \mathtt X^{-1}(\bar t,\bar x) & \text{otherwise}.
\end{cases}
\end{equation*}
Define the Lipschitz curves
\begin{equation*}
\gamma^-_{(\bar t,\bar x)}(t) := \inf \big\{ \mathtt X(t,s): s \in \mathtt S(\bar t,\bar x) \big\}, \quad \gamma^+_{(\bar t,\bar x)}(t) := \sup \big\{ \mathtt X(t,s): s \in \mathtt S(\bar t,\bar x) \big\}.
\end{equation*}
The existence of these curves is a consequence of the fact that $\mathtt X$ (as a maximal monotone function) is surjective.
Notice that since the interval of existence is an open set, then these curves can be prolonged for $t > \bar t$.

\begin{theorem}
\label{T_regul_jumps}
For every point $(\bar t,\bar x)$ it holds
\begin{equation*}
\lim_{\substack{(t,x) \to (\bar t,\bar x) \\ x < \gamma^-_{\bar t,\bar x}(t)}} \big| u(t,x) - u(\bar t,\bar x-) \big| = 0, \quad \lim_{\substack{(t,x) \to (\bar t,\bar x) \\ x > \gamma^+_{\bar t,\bar x}(t)}} \big| u(t,x) - u(\bar t,\bar x+) \big| = 0.
\end{equation*}
\end{theorem}

\begin{proof}
We prove the statement for the first limit, because the analysis of the second one is completely similar.

Firs of all Lemma \ref{L_exact_l_infty_decr} implies directly that
\begin{equation*}
\lim_{\substack{(t,x) \to (\bar t,\bar x) \\ x < \gamma^-_{\bar t,\bar x}(t), t \geq \bar t}} \big| u(t,x) - u(\bar t,\bar x-) \big| = 0.
\end{equation*}
If $(t_n,x_n) \to (\bar t,\bar x)$, $t_n < \bar t$ and $x < \gamma^-_{\bar t,\bar x}(t)$, is a sequence such that
\begin{equation*}
\lim_{n \to \infty} \big| u(t_n,x_n) - u(\bar t,\bar x-) \big| > \e,
\end{equation*}
then for all $\delta > 0$ there exists $n$ such that
\begin{equation*}
(t_n.x_n) \in E^\delta := \bigg\{ (t,x) : - \delta < t - \bar t < 0, \sup_{\mathtt X^{-1}(\bar t,x)} \mathtt X(t,s) < x_n < \gamma^-_{\bar t,\bar x}(t) \bigg\}.
\end{equation*}
Since $u(t,x) \to u(t,\bar x-)$ as $x \nearrow \bar x$, Lemma \ref{L_exact_l_infty_decr} implies thus that there exists a cancellation of ate least $\e/2$ in $E^\delta$. Since $\cap_\delta E^\delta = 0$, we reach a contradiction.
%
\end{proof}

Clearly to conclude which of these points are approximate jump points of $u$ (i.e. there exist right and left $L^1$-limits across a line), one has to say more about the properties of the curves $\gamma^\pm_{(\bar t,\bar x)}$.

Let $\Theta^\mathrm{canc}, \Theta^\mathrm{inter} \subset \R^2$ be the smallest countable sets where the atomic part of the cancellation measure $\xi^\mathrm{canc} =\mathtt X_\sharp (\partial_t \chi_{t < \mathtt T})$ and of the interaction measure $\xi^\mathrm{inter} = \mathtt X_\sharp (\int \partial_t \sigma(s) ds)$ are concentrated, respectively:
\begin{equation}
\label{E_count_Theta}
\xi^\mathrm{canc,a} = \sum_{(t_k,x_k) \in \Theta^\mathrm{canc}} c_k \delta_{(t_k,x_k)} 
, \qquad \xi^\mathrm{inter,a} = \sum_{(t_\ell,x_\ell) \in \Theta^\mathrm{inter}} c_\ell \delta_{(t_\ell,x_\ell)}, 
\end{equation}
with $c_k,d_\ell > 0$.

\begin{lemma}
\label{L_Theta_on_Xi}
It holds $\Theta := \Theta^\mathrm{canc} \cup \Theta^\mathrm{inter} \subset \Xi$.
\end{lemma}

\begin{proof}
For $\Theta^\mathrm{canc}$ this is a direct consequence of Corollary \ref{C_canc_on_jumps}. Let now $(\bar t,\bar x)$ be a continuity point of $u$. Then
\begin{equation*}
\lim_{(t,x) \to (\bar t,\bar x)} \big| \tilde \lambda(t,x) - f'(u(\bar t,\bar x)) \big| = 0.
\end{equation*}
This implies also that 
\begin{equation*}
\lim_{(t,\mathtt X(t,s)) \to (\bar t,\bar x)} \big| \tilde \lambda(t,\mathtt X(t,s)) - f'(u(\bar t,\bar x)) \big| = 0,
\end{equation*}
and one can conclude by means of Proposition \ref{P_cancell_points}.
\end{proof}

We now identify the jump points of $u$, which coincide with $\Xi \setminus \Theta$, where as before $\Theta = \Theta^\mathrm{canc} \cup \Theta^\mathrm{inter}$. 

\begin{theorem}
\label{T_rect_out_inter}
It $(\bar t,\bar x) \notin \Theta$, then $(\bar t,\bar x)$ is an $L^1$-approximate jump point of $u$. Moreover, if $(1,p)$ is the approximate tangent of $\Xi$ in $(\bar t,\bar x)$, then $p = \tilde \lambda(\bar t,\bar x)$ and for all $\delta > 0$ it holds
\begin{equation}
\label{E_cts_cone}
\lim_{\substack{(t,x) \to (\bar t,\bar x) \\ x < \bar x + p(t-\bar t) - \delta |t - \bar t|}} u(t,x) = u(\bar t,\bar x-), \qquad \lim_{\substack{(t,x) \to (\bar t,\bar x) \\ x > \bar x + p(t-\bar t) + \delta |t - \bar t|}} u(t,x) = u(\bar t,\bar x+).
\end{equation}
\end{theorem}

\begin{proof}
It is clear that \eqref{E_cts_cone} imply the first, so we will only prove the second part of the statement.

From the previous lemma we have only to consider a point $(\bar t,\bar x)$ such that (for definiteness)
\begin{equation*}
u(\bar t,\bar x-) < u(\bar t,\bar x+).
\end{equation*}
Since $(\bar t,\bar x) \notin \Theta^\mathrm{inter}$, it follows that for $\mathcal L^1$-a.e. $s \in \mathtt S(\bar t,\bar x)$ it holds
\begin{equation*}
\lim_{t \to \bar t} \sigma(t,s) = \sigma(\bar t,s).
\end{equation*}
Since $s \mapsto \mathtt X(t,s)$ is monotone, the only possibility is that $\sigma(\bar t,s) = \tilde \lambda(\bar t,\bar x)$.

Assume now by contradiction that there exists a sequence $(t_n,x_n)$ such that
\begin{equation*}
(t_n,x_n) \to (\bar t,\bar x), \qquad x_n < \tilde \lambda(\bar t,\bar x) (t_n-\bar t) + \delta |t_n-\bar t|
\end{equation*}
and
\begin{equation*}
\lim_{n \to \infty} \big| u(t_n,x_n) - u(\bar t,\bar x-) \big| > \epsilon.
\end{equation*}
Since all the curves $t \mapsto \mathtt X(t,s)$, $s \in \mathtt X^{-1}(\bar t,\bar x)$ are tangent with derivative $\tilde \lambda(\bar t,\bar x)$, then it follows that the in a sufficiently small neighborhood $O$ of $(\bar t,\bar x)$ it holds
\begin{equation*}
\Big\{ \mathtt X(t,s): t \in [0,\mathtt T(s)], s > \inf \mathtt X^{-1}(\bar t,\bar x) + \e/2 \Big\} \cap O \cap \big\{ x < \tilde \lambda(\bar t,\bar x) (t-\bar t) + \delta |t-\bar t| \big\} = \emptyset.
\end{equation*}
A completely similar argument to the one used in the proof of Theorem \ref{T_regul_jumps} implies that in the set O a cancellation of order $\e/4$, and by shrinking $O$ to $(\bar t,\bar x)$ we obtain a contradiction.

The proof for $x > \bar x + \tilde \lambda(\bar t,\bar x) (t-\bar t) + \delta |t - \bar t|$ is analogous.
\end{proof}

\subsection{Examples}
\label{S_examples}

In order to show that our results are almost optimal, and that the situation in this general settings is more complicated that in the standard genuinely nonlinear situation, we recall the example given in Section 6 of \cite{BY-global}.

%
%
The example shows the existence of a $2 \times 2$ system of the form
\begin{equation*}
\begin{cases}
u_t+f(u,v)_x=0,\\
v_t-v_x=0,
\end{cases}
\end{equation*}
such that for carefully chosen initial data $(u_0,v_0)$ the wave pattern is as in Figure \ref{Fi_wave_pattern}: the jump set of the first component $u$ is the line $x=0$, and actual jump set (i.e. the points where $u(t)$ is discontinuous) is a Cantor compact set $K \times \{0\}$ of positive $\mathcal H^1$-measure. This can be done by constructing some suitable $f$ to be convex w.r.t $u$ when $v$ is positive and concave w.r.t $u$ when $v$ negative.

The behavior of $u$ in the complementary open set is shown in Figure \ref{Fi_construction_f}: due to transversal interactions with $v$, first a rarefaction wave open the jump, then it becomes a compression wave, and then it becomes again a shock.

\begin{figure}
\begin{minipage}{5cm}
\resizebox{5cm}{!}{
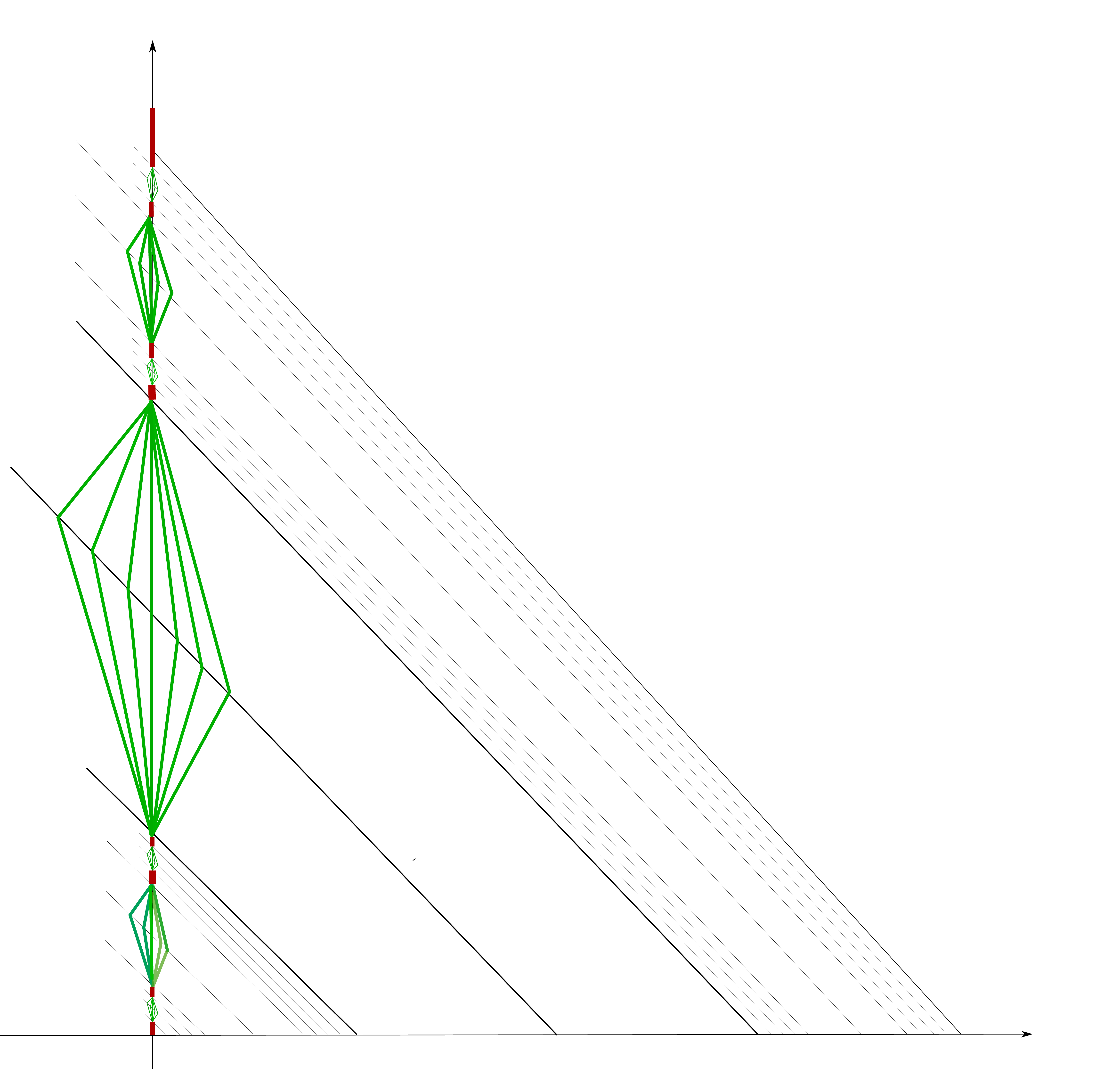}
\caption{}
\label{Fi_wave_pattern}
\end{minipage}
\begin{minipage}{7cm}
\resizebox{7cm}{!}{
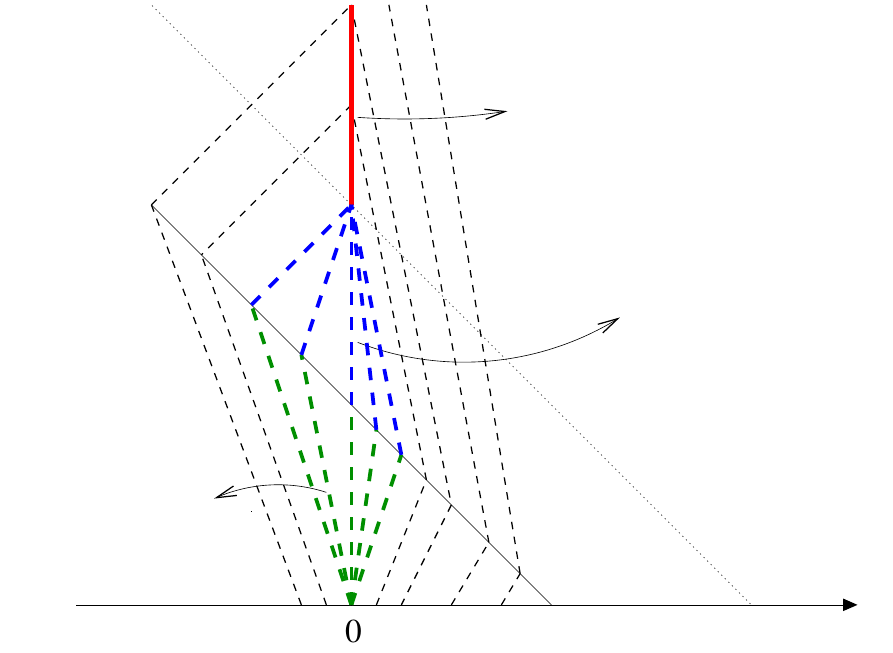}
\caption{}
\label{Fi_construction_f}
\end{minipage}
\end{figure}

It is fairly so see that the the curves $\gamma^\pm_{(\bar t,\bar x)}$ (with $\bar x = 0$ here) cannot coincide on an arbitrarily small neighborhood of every point $\bar t \in K$: in fact, in every neighborhood of $\bar t$ there is an open set of the form shown in Figure \ref{Fi_construction_f}. Thus, at least for general systems, a result like Theorem \ref{T_regul_jumps} cannot be improved.

\subsubsection{Cancellation vs interaction measure}
\label{Sss_canc_vs_inter}

Here we show that the cancellation measure $\mathtt X_\sharp (\partial_t \chi)$ is not a.c. w.r.t. the interaction measure $\mathtt X_\sharp (\int \partial_t \sigma(s) ds )$. 

\begin{figure}
\begin{minipage}{7cm}
\resizebox{7cm}{!}{
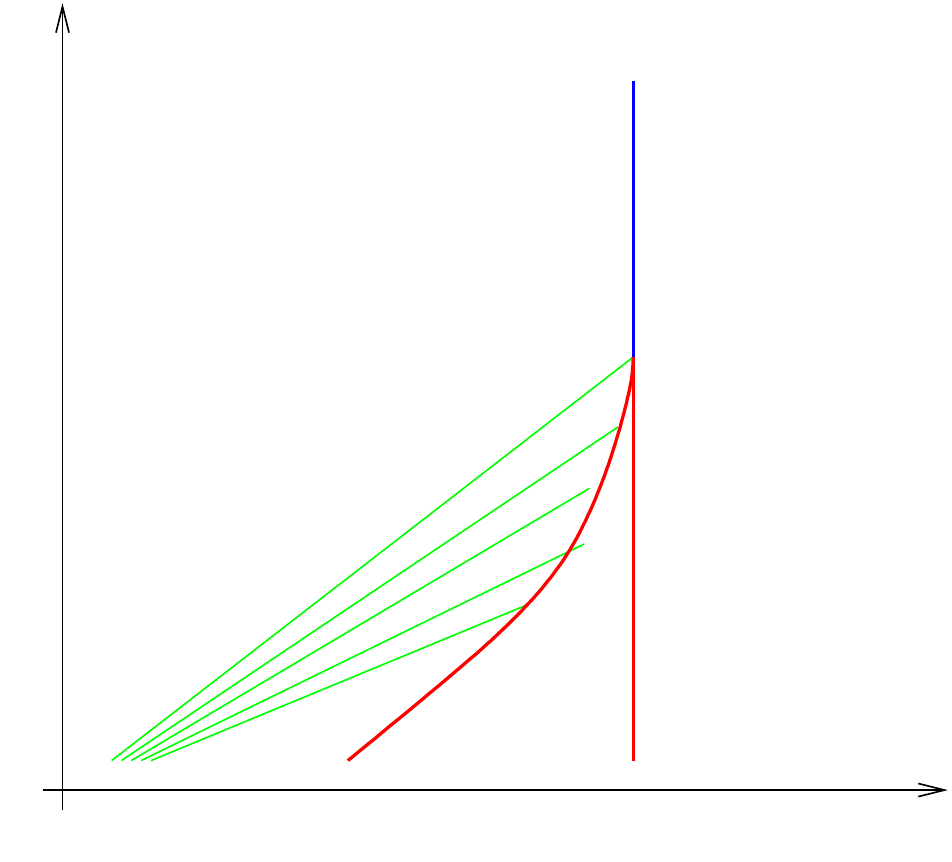}
\caption{}
\label{Fi_canc_vs_inter_1}
\end{minipage}
\begin{minipage}{7cm}
\resizebox{7cm}{!}{
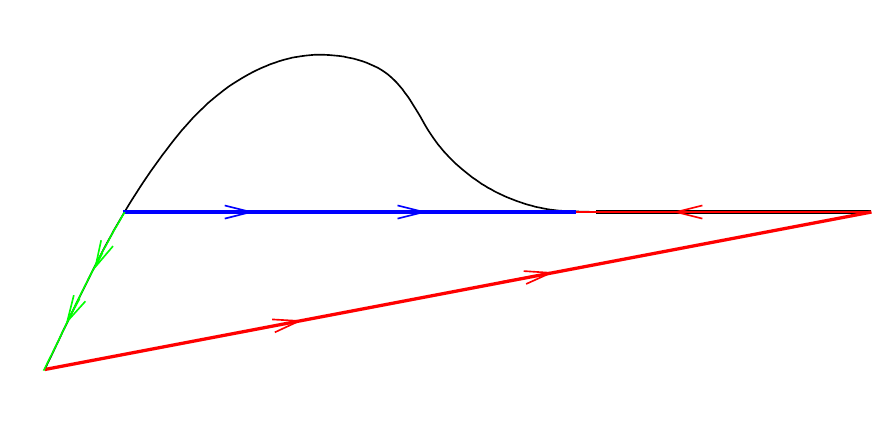}
\caption{}
\label{Fi_canc_vs_inter_2}
\end{minipage}
\end{figure}

In Figure \ref{Fi_canc_vs_inter_1} it is shown a wave pattern produced by the flux function represented in Figure \ref{Fi_canc_vs_inter_2} with suitable initial conditions: two jumps of different sign meet at a given point with the same speed. This is achieved by slowing down the shock $[u^-_m,u^+_m]$ by a rarefaction.

It follows that in the point $(\bar t,\bar x)$ a cancellation of size $2(u^+_m - u+)$ occurs, but the interaction measure is $0$.

\subsubsection{\texorpdfstring{Points outside $\Theta$ in which the curves $\gamma^\pm$ are not tangent}{Points not of strong interaction or cancellation which the curves are not tangent}}
\label{Sss_not_jump}

The final example shows that even if no strong interaction or cancellation occurs in a point $(\bar t,\bar x) \in \Xi$, then the curves $\gamma^\pm_{(\bar t,\bar x)}$ may be not tangent in $(\bar t,\bar x)$.

\begin{figure}
\begin{minipage}{7cm}
\resizebox{5cm}{!}{
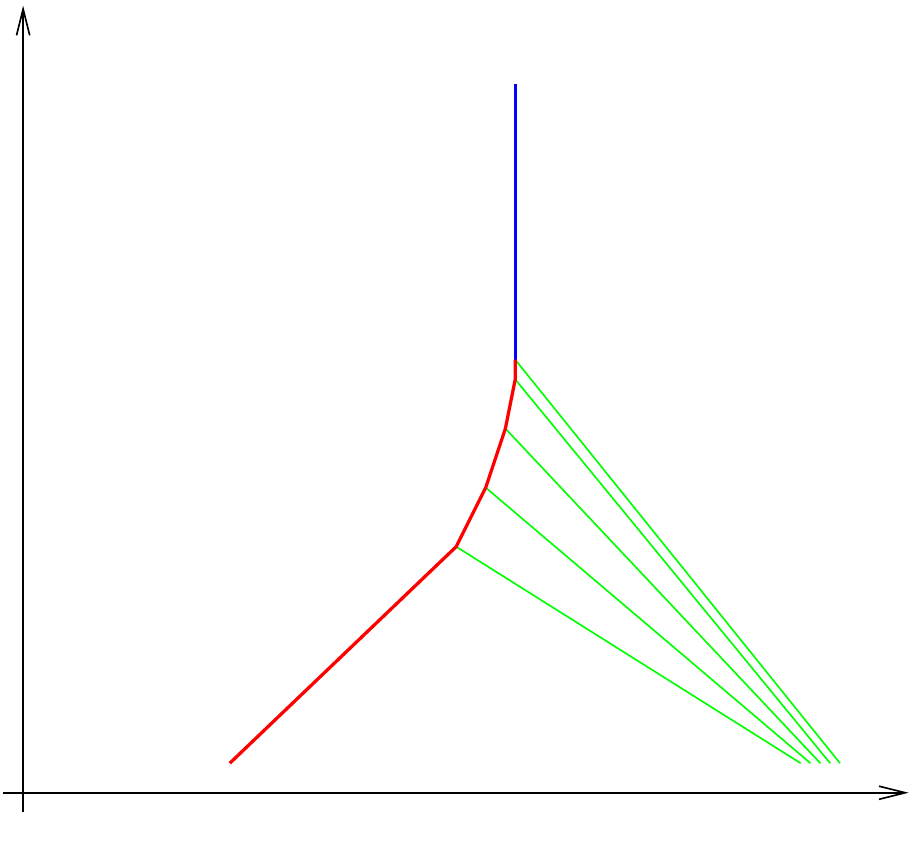}
\caption{}
\label{Fi_can_inter_jump_2}
\end{minipage}
\begin{minipage}{7cm}
\begin{center}
\resizebox{7cm}{!}{
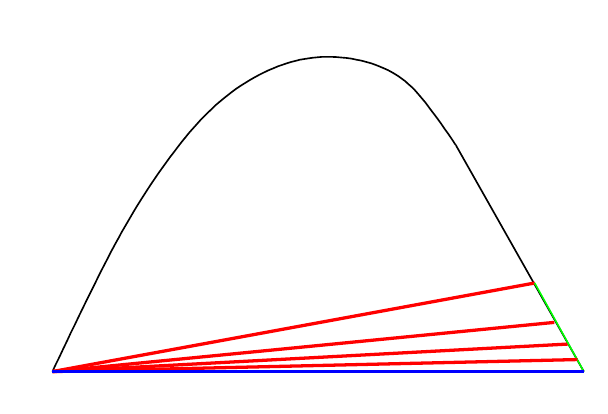}
\caption{}
\label{Fi_can_inter_jump_1}
\end{center}
\end{minipage}
\end{figure}

If Figure \ref{Fi_can_inter_jump_2} it is shown a wave pattern produced by the flux function represented in Figure \ref{Fi_can_inter_jump_1} with suitable initial conditions: a countable family of jumps approaching from the right and whose interaction points converge to $(\bar t,\bar x)$. The size of the these jumps converges to $0$, so that $(\bar t, \bar x) \notin \Theta$.

It is fairly easy to see that:
\begin{enumerate}
\item $(\bar t,\bar x)$ is a jump points of $u$;
\item the curves $\gamma^\pm_{(\bar t,\bar x)}$ are not parallel at $(\bar t,\bar x)$.
\end{enumerate}

This example thus justifies the analysis of Theorem \ref{T_rect_out_inter}, particular the parameter $\delta$ appearing in \eqref{E_cts_cone}. 


\bibliographystyle{siam}
\bibliography{ref}

\end{document}